\documentclass[11pt]{amsart}

\usepackage{amssymb,amsfonts}
\usepackage{graphicx}
\usepackage[all,arc]{xy}
\usepackage{enumerate}
\usepackage{hyperref}
\usepackage{mathrsfs}
\usepackage{tikz-cd}
\usepackage[left=1in,top=1in,right=1in,bottom=1in]{geometry}

\newtheorem{thm}{Theorem}[section]
\newtheorem*{thm*}{Theorem}

\newtheorem{cor}[thm]{Corollary}
\newtheorem{prop}[thm]{Proposition}
\newtheorem{lem}[thm]{Lemma}
\newtheorem*{lemma}{Lemma}

\theoremstyle{definition}
\newtheorem{defn}[thm]{Definition}

\newtheorem{exmp}[thm]{Example}

\newtheorem{notn}[thm]{Notation}

\newtheorem{rem}[thm]{Remark}

\theoremstyle{remark}

\newcommand{\Z}{\mathbb{Z}}
\newcommand{\Q}{\mathbb{Q}}
\newcommand{\G}{\mathbf{G}}
\renewcommand{\P}{\mathbb{P}}
\renewcommand{\exp}{\text{exp}}

\newcommand{\trop}{\text{trop}}

\renewcommand{\d}{\partial}

\newcommand{\B}{\mathbf{B}}
\newcommand{\R}{\mathbf{R}}
\newcommand{\D}{\mathcal{D}}

\newcommand{\val}{\operatorname{val}}
\newcommand{\Aut}{\operatorname{Aut}}

\renewcommand{\H}{\mathbf{H}}
\renewcommand{\emptyset}{\varnothing}
\newcommand{\T}{\mathbf{T}}

\newcommand{\V}{\mathsf{V}}

\renewcommand{\SS}{\mathfrak{S}}

\makeatletter
\let\c@equation\c@thm
\makeatother
\numberwithin{equation}{section}

\usepackage[toc]{appendix}

\graphicspath{ {hassett-figs/} }

\title{Automorphisms of tropical Hassett spaces}

\author{Sam Freedman}
\author{Joseph Hlavinka}
\author{Siddarth Kannan}
\address{Department of Mathematics, Brown University, Providence, RI 02906}
\email[S. Freedman]{\url{samuel_freedman@brown.edu}}

\email[J. Hlavinka]{\url{joseph_hlavinka@brown.edu}}

\email[S. Kannan]{\url{siddarth_kannan@brown.edu}}

\begin{document}
	
	\begin{abstract}
	Given an integer $g \geq 0$ and a weight vector $w \in \Q^n \cap (0, 1]^n$ satisfying $2g - 2 + \sum w_i > 0$, let $\Delta_{g, w}$ denote the moduli space of $n$-marked, $w$-stable tropical curves of genus $g$ and volume one. We calculate the automorphism group $\Aut(\Delta_{g, w})$ for $g \geq 1$ and arbitrary $w$, and we calculate the group $\Aut(\Delta_{0, w})$ when $w$ is \textit{heavy/light}. In both of these cases, we show that $\Aut(\Delta_{g, w}) \cong \Aut(K_w)$, where $K_w$ is the abstract simplicial complex on $\{1, \ldots, n\}$ whose faces are subsets with $w$-weight at most $1$. We show that these groups are precisely the finite direct products of symmetric groups. The space $\Delta_{g, w}$ may also be identified with the dual complex of the divisor of singular curves in the algebraic Hassett space $\overline{\mathcal{M}}_{g, w}$. Following the work of Massarenti and Mella ~\cite{MassMella} on the biregular automorphism group $\Aut(\overline{\mathcal{M}}_{g, w})$, we show that $\Aut(\Delta_{g, w})$ is naturally identified with the subgroup of automorphisms which preserve the divisor of singular curves. 
	\end{abstract}
	\maketitle
	\tableofcontents
	\section{Introduction}
	
	Fix integers $g, n \geq 0$ such that $2g - 2 + n > 0$, let $\mathcal{M}_{g, n}$ denote the moduli stack of smooth $n$-marked algebraic curves of genus $g$, and let $\overline{\mathcal{M}}_{g, n}$ denote its Deligne-Mumford-Knudsen compactification by stable curves. Brendan Hassett ~\cite{hassett} has given a large family of alternate modular compactifications of $\mathcal{M}_{g, n}$: given a weight vector $w \in \Q^n \cap (0, 1]^n$ satisfying
	\[2g - 2 + \sum_{i = 1}^{n} w_i > 0, \]
    Hassett constructs a smooth and proper Deligne-Mumford moduli stack $\overline{\mathcal{M}}_{g, w}$, birational to $\overline{\mathcal{M}}_{g, n}$, which contains $\mathcal{M}_{g, n}$ as a dense open substack. The points of $\overline{\mathcal{M}}_{g, w}$ represent $n$-pointed nodal curves $(C, p_1, \ldots, p_n)$, satisfying (i) that the $\Q$-divisor $K_C + \sum w_i p_i$ is ample along each component of $C$, where $K_C$ is the canonical divisor of $C$, and (ii) if $p_{i_{1}} = \cdots = p_{i_{r}}$, then $w_{i_1} + \cdots + w_{i_r} \leq 1$. In particular, when $w = (1^{(n)})$ is the all $1$'s vector, we have an equality $\overline{\mathcal{M}}_{g, w} = \overline{\mathcal{M}}_{g, n}$.
    
    An important feature of the compactifaction $\mathcal{M}_{g, n} \subset \overline{\mathcal{M}}_{g, n}$ is that the boundary divisor \[\d\overline{\mathcal{M}}_{g, n} := \overline{\mathcal{M}}_{g, n} \smallsetminus \mathcal{M}_{g, n}\] is normal crossings. In ~\cite{CGP1}, Chan, Galatius, and Payne, following work of Harper ~\cite{Harper} and Abramovich-Caporaso-Payne ~\cite{ACP}, show how to construct the dual complex $\Delta(\mathcal{X}, \mathcal{D})$ of a normal crossings divisor $\mathcal{D}$ on a Deligne-Mumford stack $\mathcal{X}$. They study $\Delta(\mathcal{X}, \mathcal{D})$ in the case where $\mathcal{X} = \overline{\mathcal{M}}_{g, n}$ and $\mathcal{D} = \d\overline{\mathcal{M}}_{g, n}$, showing that $\Delta(\mathcal{X}, \mathcal{D}) = \Delta_{g, n}$ is identified with the link of the cone point in the moduli space $M_{g, n}^\trop$ of stable $n$-marked tropical curves of genus $g$.

    
    On the other hand, the complement of $\mathcal{M}_{g, n}$ in Hassett's compactification $\overline{\mathcal{M}}_{g, w}$ is not in general normal crossings. However, if we put $\mathcal{M}_{g, w}$ for the locus of smooth, but not necessarily distinctly marked, curves in $\overline{\mathcal{M}}_{g, w}$, then the complement
    \[\d \overline{\mathcal{M}}_{g, w} : = \overline{\mathcal{M}}_{g, w} \smallsetminus \mathcal{M}_{g, w} \]
    has normal crossings, and the resulting dual intersection complex $\Delta_{g, w}$ is the link of the cone point in the moduli space $M_{g, w}^\trop$ of $n$-marked, $w$-stable tropical curves of genus $g$, as has been established by Ulirsch ~\cite{Ulirsch}.
    
    In this paper, we are interested in the automorphism groups of the complexes $\Delta_{g, w}$, taken in the category of \textit{symmetric $\Delta$-complexes}, as defined in ~\cite{CGP1} and recalled in Section \ref{construction}. Given a weight vector $w$, we can form an abstract simplicial complex $K_w$ with vertex set $\{1, \ldots, n\}$ by declaring that a subset $S \subseteq \{1, \ldots, n\}$ belongs to $K_w$ if and only if $\sum_{i \in S} w_i \leq 1$; this construction was considered by Alexeev and Guy ~\cite{alexeev} in their work on moduli of weighted stable maps. See Figure \ref{WeightComplexExamples} for some examples of the complex $K_w$. Our first main theorem determines $\Aut(\Delta_{g, w})$ in terms of $K_w$ for $g \geq 1$.
    \begin{thm}\label{Auts}
    Let $g \geq 1$ and suppose $w \in \Q^n \cap (0, 1]^n$ for some $n$ such that $2g - 2 + n \geq 3$. Then
    \[\Aut(\Delta_{g, w}) \cong \Aut(K_w), \]
    where $\Aut(K_w)$ acts by permuting the markings.
    \end{thm}
    Here $\Aut(K_w)$ is viewed as a subgroup of $S_n := \mathrm{Perm}(\{1, \ldots, n\})$. Theorem \ref{Auts} will be proven in Section \ref{ProofofAuts}, and the failure of the $g = 0$ case will be further explored and partially remedied in Section \ref{genuszero}.  Following Cavalieri, Hampe, Markwig, and Ranganathan ~\cite{CHMR2014moduli}, we refer to weight vectors satisfying the hypotheses of the following theorem as \textit{heavy/light}, with $m$ light markings and $n$ heavy markings.
    
    \begin{thm}\label{GenusZero}
    Suppose $n, m \geq 2$, with $n + m \geq 5$, and put $w = (\varepsilon^{(m)},1^{(n)})$ where $\varepsilon \leq 1/m$. Then we have
    \[\Aut(\Delta_{0, w}) \cong \Aut(K_w) \cong S_m \times S_n. \]
    \end{thm}
    
    Heavy/light Hassett spaces are of particular interest: they are also studied in \cite{Bergstrom}, \cite{Chaudhuri}, \cite{KKL}, \cite{Modular}, and \cite{MOP}.

    \begin{figure}[h]
    \centering
    \includegraphics[scale=1.3]{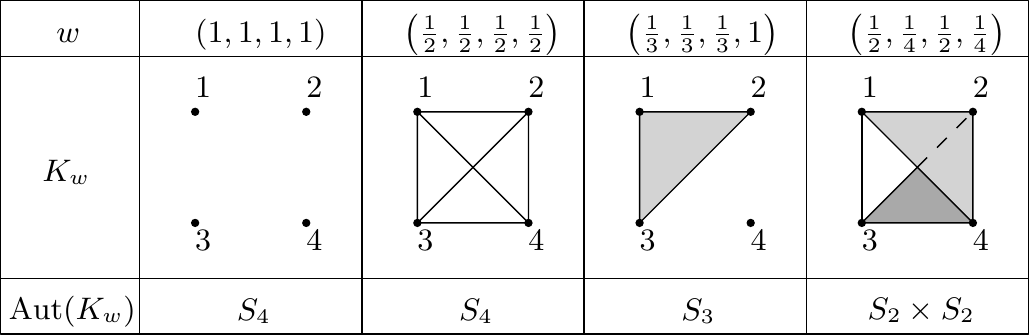}
    \caption{Examples of the simplicial complex $K_w$.}
    \label{WeightComplexExamples}
    \end{figure}
     It is also interesting to characterize the groups $\Aut(K_w)$, as in the following theorem. Since it is independent from the rest of the paper, its proof is found in Appendix \ref{ProofofAllProducts}.
    \begin{thm}\label{AllProductsArise}
    Let $G$ be a group. Then there exists $n \geq 1$ and $w \in \Q^n \cap (0, 1]^n$ such that
    \[\Aut(K_w) \cong G  \]
    if and only if $G$ is isomorphic to the direct product of finitely many symmetric groups.
    \end{thm}
    
    \subsection{Comparison with the algebraic moduli space}
    For $g, n \geq 0$ such that $2g - 2 + n \geq 3$, we have isomorphisms
    \[\Aut(\overline{\mathcal{M}}_{g, n}) \cong \Aut(\Delta_{g, n}) \cong S_n, \]
     following the results of ~\cite{K} and ~\cite{Mass}; here $S_n$ acts by relabelling the marked points. The analogous result cannot be true for general weight vectors. Indeed, ~\cite[Corollary 4.7]{hassett} states that if $w_i \leq w_i'$ for all $i$ and the complexes $K_w, K_{w'}$ coincide outside of their $1$-skeletons, then there is an isomorphism of coarse moduli spaces $\overline{M}_{g, w} \cong \overline{M}_{g, w'}$. Moreover, by ~\cite[Theorem 3.20]{MassMella}, the automorphism groups of the stacks and coarse spaces agree. This implies, for example, that when $w = (1^{(n)}, 1/2^{(m)})$, we have $\Aut(\overline{\mathcal{M}}_{g, w}) \cong S_{n + m}$. On the other hand, Theorem \ref{Auts} states that $\Aut(\Delta_{g, w}) \cong \Aut(K_w) \cong S_n \times S_m$.
    
    In ~\cite{MassMella}, Massarenti and Mella prove that for $g, n \geq 1$ such that $2g + 2 + n \geq 3$, the automorphism group of the moduli stack $\overline{\mathcal{M}}_{g, w}$ is given by the subgroup of $S_n$ generated by \emph{admissible transpositions}. These are transpositions $(i, j)$ such that, for all $S \subseteq \{1, \ldots, n\}$ with $|S| \geq 2$, we have
    \[w_i  + w(S) \leq 1 \iff w_j + w(S) \leq 1, \]
    where for a subset $S \subseteq \{1, \ldots, n\}$ we define
    \[w(S) := \sum_{i \in S} w_i. \]
    The group generated by admissible transpositions acts on $\overline{\mathcal{M}}_{g, w}$ by relabelling the marked points, and contracting rational components which become unstable if necessary. We now show that $\Aut(K_w)$ is the subgroup of $\Aut(\overline{\mathcal{M}}_{g, w})$ which preserves the locus $\partial\overline{\mathcal{M}}_{g, w}$ of singular curves.
    \begin{lem}\label{BoundaryPreserving}
    Suppose $g, n \geq 1$ with $2g - 2 + n \geq 3$, and fix $w \in \Q^n \cap (0, 1]^n$. Then
    \[\Aut(\overline{\mathcal{M}}_{g, w}, \partial\overline{\mathcal{M}}_{g, w}) \cong \Aut(K_w), \]
    where $\Aut(K_w)$ acts by permuting the markings.
    \end{lem}
    \begin{proof}
   Suppose first that $\sigma$ is in the subgroup of $S_n$ generated by admissible transpositions but $\sigma \notin \Aut(K_w)$. Then there exists some $S \subseteq \{1, \ldots, n\}$ with $|S| = 2$, $w(S) > 1$, but $w(\sigma(S)) \leq 1$; say $S = \{i, j \}$. Consider a pointed nodal curve $(C, p_1, \ldots, p_n)$ of arithmetic genus $g$ with two irreducible components $T_1, T_2$, so that $T_2$ is isomorphic to $\P^1$ and supports the marked points $p_i, p_j$, while the other marked points are distributed distinctly on $T_1$. Then $\sigma \cdot (C, p_1, \ldots, p_n)$ is obtained from $(C, p_1, \ldots, p_n)$ by first permuting the marked points according to $\sigma$, and then contracting the component $T_2$ to a point so that $p_{\sigma(i)} = p_{\sigma(j)}$ (this is necessary because $w_{\sigma(i)} + w_{\sigma(j)} \leq 1$). In particular $\sigma \cdot (C, p_1, \ldots, p_n)$ is no longer a singular curve. This shows that $\Aut(\overline{\mathcal{M}}_{g, w}, \partial\overline{\mathcal{M}}_{g, w})$ is a subgroup of $\Aut(K_w)$; to finish, we simply note that when applying $\sigma \in \Aut(K_w)$ to a nodal curve $(C, p_1, \ldots, p_n)$, there is never a need to contract any components, so $\Aut(K_w)$ preserves the boundary.
    \end{proof}
    The simplicial complexes $K_w$ correspond to the chambers of the \textit{fine chamber decomposition} of ~\cite[Section 5]{hassett} (see also \cite[Section 2]{alexeev}).
    
    In general, if $\mathcal{D}$ is a normal crossings divisor on a variety or DM stack $\mathcal{X}$, one has a homomorphism
    \[\Aut(\mathcal{X}, \mathcal{D}) \to \Aut(\Delta(\mathcal{X}, \mathcal{D})), \]
    where $\Delta(\mathcal{X}, \mathcal{D})$ is the dual complex of $\mathcal{D}$ in $\mathcal{X}$. Given Lemma \ref{BoundaryPreserving}, the upshot of Theorem \ref{Auts} is that this map is an isomorphism if we specialize to $\mathcal{D} = \partial \overline{\mathcal{M}}_{g, w}$ and $\mathcal{X} = \overline{\mathcal{M}}_{g, w}$:
    \begin{cor}
    Suppose $g, n \geq 1$ with $2g - 2 + n \geq 3$, and $w \in \Q^n \cap (0, 1]^n$. Then the map
    \[\Aut(\overline{\mathcal{M}}_{g, w}, \partial\overline{\mathcal{M}}_{g, w}) \to \Aut(\Delta_{g, w}) \]
    is an isomorphism.
    \end{cor}
    We can also give a sufficient condition for the groups $\Aut(\overline{\mathcal{M}}_{g, w})$ and $\Aut(K_w)$ to coincide. Recall that a \textit{facet} of a simplicial complex is a face that is maximal with respect to inclusion.
    \begin{cor}\label{AlgebraicVTropical}
    Suppose $g, n \geq 1$ with $2g - 2 + n \geq 3$, and $w \in \Q^n \cap (0, 1]^n$. Then, if $K_w$ has no $1$-dimensional facets, then
    \[ \Aut(\overline{\mathcal{M}}_{g, w}) = \Aut(K_w).  \]
    \end{cor}
    \begin{proof}
    It suffices to show that every admissible transposition $\tau = (i, j)$ is in $\Aut(K_w)$. Suppose $S \subseteq \{1, \ldots, n \}$ satisfies $w(S) \leq 1$, in order to show that $w(\tau(S)) \leq 1$. If both $i, j \in S$ or $i, j \in S^c$, then $w(\tau(S)) = w(S)$, so we suppose without loss of generality that $i \in S$ while $j \in S^c$. Then, if $|S| \geq 3$, by the definition of admissible transposition, we must have $w(\tau(S)) \leq 1$. If $|S| = 2$, then $S$ forms a $1$-simplex of $K_w$, and cannot be a facet. Thus there exists some $T \subseteq \{1, \ldots, n \}$ such that $S \subsetneq T$ and $w(T) \leq 1$. Since $|T| \geq 3$, we have $w(\tau(T)) \leq 1$, and $w(\tau(S)) < w(\tau(T))$, so $w(\tau(S)) \leq 1$ and $\tau \in \Aut(K_w)$, finishing the proof.
    \end{proof}
    The condition of Corollary \ref{AlgebraicVTropical} is sufficient but not necessary: indeed, if $w = (1/2^{(n)})$, then $\Aut(\overline{\mathcal{M}}_{g, w}) \cong \Aut(K_w) \cong S_n$, but $K_w$ is the complete graph on $n$ vertices, so all of its facets are $1$-dimensional.

    \subsection{Tropical Hassett spaces excluded by Theorem \ref{Auts}} When $g \geq 1$, the space $\Delta_{g, w}$ is nonempty as long as $3g - 3 + n > 0$, so the positive genus cases not covered by Theorem \ref{Auts} are $(g, n) = (1, 1), (1, 2)$. When $n = 1$ we have $\Delta_{1, w} = \Delta_{1, 1}$ for any $w$, and this space is a single point, so the automorphism group is trivial. When $n = 2$, so $w = (w_1, w_2)$, we have $\Delta_{1, w} \cong \Delta_{1, 2}$ if $w_1 + w_2 > 1$, so in this case the automorphism group is trivial by ~\cite[Example 2.19]{K}. When $w_1 + w_2 \leq 1$, $\Delta_{1, w}$ will be shown to be trivial in Example \ref{ExceptionalCase}.
    
    \subsection{Related work} 
    In the special case $w = (1^{(n)})$, the automorphism group of $\Delta_{g, w}$ is known to be $S_n$: this is due to Abreu and Pacini ~\cite{AbreuPacini} when $g = 0$, and to the third author ~\cite{K} in arbitrary genus. Indeed, one of the main technical theorems in ~\cite{K} is also the driving force behind the calculation in the current paper.
    
    The topology of $\Delta_{g, w}$ was studied for $g \leq 1$ by Cerbu \textit{et al.} in ~\cite{CMPRS}. When $g = 0$ and the weight vector $w$ has at least two entries equal to $1$, the space $\Delta_{0, w}$ has the homotopy type of a wedge of spheres, possibly of varying dimension. Closed formulas for the number of spheres are known when $w$ is heavy/light. In higher genus, the topology of $\Delta_{g, w}$ has been partially explored by Li, Serpente, Yun, and the third author in ~\cite{kannan2020topology}. When $g \geq 1$, and for any value of $w$, the space $\Delta_{g, w}$ is shown to be simply-connected. Formulas for the Euler characteristic of $\Delta_{g, w}$ in terms of the combinatorics of the complex $K_w$ have also been derived. 
    
    The cone complexes $M_{0,w}^{\trop}$ were studied in the context of tropical compactification in \cite{CHMR2014moduli}. The authors showed that the complex $M_{0, w}^\trop$ can be embedded as a balanced fan $\Sigma_{0, w}$ in a real vector space if and only if $w$ is heavy/light. In the heavy/light case, they show that the locus $\mathcal{M}_{0, w}$ embeds into the toric variety $X(\Sigma_{0, w})$, in such a way that taking the closure of the image gives Hassett's original compactification. This procedure gives an isomorphism of Chow rings $A^*(\overline{\mathcal{M}}_{0, w}) \cong A^*(X(\Sigma_{0, w}))$, allowing for the computation of $A^*(\overline{\mathcal{M}}_{0, w})$ carried out in ~\cite{KKL}.
    

    \subsection{Acknowledgements}
    JH was supported by a BrownConnect Collaborative SPRINT Award. SK was supported by an NSF Graduate Research Fellowship.

    \section{Graphs and $\Delta_{g, w}$}\label{construction}

    We first recall the category $\Gamma_{g, n}$ of weighted stable graphs of genus $g$; see ~\cite[\S 2.1]{CGP2} or ~\cite[\S 2.1]{K} for a precise definition. An object of $\Gamma_{g, n}$ is a triple $\G = (G, h, m)$ where $G$ is a finite connected graph, while $h: V(G) \to \Z_{\geq 0}$ and $m: \{1,\ldots, n\} \to V(G)$ are functions; these three data are required to satisfy
    \[b^1(G) + \sum_{v \in V(G)} h(v) = g, \]
    and
    \[2 h(v) - 2 + \val(v) + |m^{-1}(v)| > 0  \]
    for all $v \in V(G)$. In the above, $b^1(G) = |E(G)| - |V(G)| + 1$ denotes the first Betti number of $G$, and $\val(v)$ denotes the valence of the vertex $v$, which is the number of half-edges emanating from $v$. A \textit{\textbf{morphism}} of weighted stable graphs of genus $g$ is a composition of isomorphisms and edge-contractions. Given a morphism $\varphi: \G \to \G'$ in $\Gamma_{g, n}$, each edge in $\G'$ has a unique preimage in $\G$. We write $\varphi^*: E(\G') \to E(\G)$ for the induced map of sets.
    \begin{defn}
    Given $w \in \Q^n \cap (0, 1]^n$, say $\G \in \mathrm{Ob}(\Gamma_{g, n})$ is \textit{\textbf{$w$-stable}} if for all $v \in V(\G)$, we have
    \[ 2 h(v) - 2 + \val(v) + w(m^{-1}(v)) > 0. \]
    We write $\Gamma_{g, w}$ for the full subcategory of $\Gamma_{g, n}$ whose objects are those which are $w$-stable.
    \end{defn}
    We remark that when $w = (1^{(n)})$, we have $\Gamma_{g, w} = \Gamma_{g,n}$. As in \cite{K}, it is useful to define an auxiliary groupoid $\Gamma_{g, w}^{\mathrm{EL}}$, whose objects are edge-labelled $w$-stable graphs of genus $g$.
    \begin{defn}
    Define $\Gamma_{g, w}^{\mathrm{EL}}$ to be the groupoid of pairs $(\G, \tau)$ where $\G \in \mathrm{Ob}(\Gamma_{g,w})$ and $\tau: E(\G) \to [p]$ is a bijection, where for an integer $p \geq 0$, we define
    \[[p] = \{0, \ldots, p\}. \]
    An isomorphism of pairs $\varphi: (\G, \tau) \to (\G', \tau')$ is an isomorphism $\G \to \G'$ such that the diagram
    \[
    \begin{tikzcd}
    &E(\G') \arrow[rr, "\varphi^*"] \arrow[dr, "\tau'"] & &E(\G) \arrow[dl, "\tau"'] \\
    & &\lbrack p \rbrack &
    \end{tikzcd}
    \]
    commutes.
    \end{defn}
    
    An open problem in graph theory is to classify those graphs which are determined, up to isomorphism, by their deck of edge-contractions. The reader may consult the thesis of Antoine Poirier ~\cite{Poirier} for a thorough overview of this problem. The main technical tool of this paper is a solution to an easier version of this problem for the categories $\Gamma_{g, w}^{\mathrm{EL}}$. Given $(\G, \tau : E(\G) \to [p]) \in \mathrm{Ob}(\Gamma_{g, w}^{\mathrm{EL}})$ and $i \in [p]$, we set $e_i = \tau^{-1}(i) \in E(\G)$, and put $\tau_i : E(\G) \to [p]$ for the unique edge-labelling making the diagram 
     \[
    \begin{tikzcd}
    &E(\G/e_i) \arrow[r, "c_i^*"] \arrow[d, "\tau_i"] &E(\G) \arrow[d, "\tau"] \\
    &\lbrack p - 1 \rbrack \arrow[r, "\delta^i"] & \lbrack p \rbrack
    \end{tikzcd}
    \]
    commute, where $c_i: \G \to \G/e_i$ is the contraction of edge $e_i$ and $\delta^i: [p - 1] \to [p]$ is the unique order-preserving injection whose image does not contain $i$.
    \begin{defn}
    Let $(\G, \tau) \in \mathrm{Ob}(\Gamma_{g, w}^{\mathrm{EL}})$. We define the \textit{\textbf{nonloop contraction deck}} of $(\G, \tau)$ to be the set of pairs
    \[\D^\G_{\tau} := \{((\G/e_i, \tau_i), i) \mid e_i \text{ is not a loop of }\G \} \subseteq \mathrm{Ob}(\Gamma_{g, w}^{\mathrm{EL}}) \times [p].  \]
    \end{defn}
    Given two lists $\D_1 = \{((\G_i, \tau_i), i) \mid i \in J_1 \}, \D_2 = \{((\H_i, \pi_i), i) \mid i \in J_2 \}$ of $\Gamma_{g, w}^{\mathrm{EL}}$ objects indexed by $J_1, J_2 \subseteq [p]$, we write $\D_1 \cong \D_2$ if $J_1 = J_2$, and $(\G_i, \tau_i) \cong (\H_i, \pi_i)$ for all $i \in J_1$.
    \begin{thm}\label{Reconstruction}
    Suppose $(\G, \tau), (\G', \tau') \in \mathrm{Ob}(\Gamma_{g, w}^{\mathrm{EL}})$ with $b_1(\G) = b_1(\G') = g$. Suppose further that $\D^{\G}_{\tau} \cong \D^{\G'}_{\tau'}$ and that $|V(\G)| = |V(\G')| \geq 3$. Then $(\G, \tau) \cong (\G', \tau')$.
    \end{thm}
    \begin{proof}
    In the case $w = (1^{(n)})$, this is Theorem 4.2 in ~\cite{K}. The case of general $w$ follows from this one, as $\Gamma_{g, w}^{\mathrm{EL}}$ may be identified with a full subcategory of $\Gamma_{g, n}^{\mathrm{EL}}$.
    \end{proof}
    \subsection{Description of $\Delta_{g, w}$ as a functor}
    
    We will calculate $\Aut(\Delta_{g, w})$ in the category of \textit{symmetric $\Delta$-complexes}, as introduced by Chan, Galatius, and Payne ~\cite{CGP1}. Put $\mathrm{I}$ for the category whose objects are the sets $[p]$ for each $p \geq 0$, and whose morphisms are all injections.
    \begin{defn}
    A \textbf{\textit{symmetric $\Delta$-complex}} is a functor $X: \mathrm{I}^{\mathrm{op}} \to \mathsf{Set}$.
    \end{defn}
    A morphism of symmetric $\Delta$-complexes is a natural transformation of functors. A symmetric $\Delta$-complex $X: \mathrm{I}^{\mathrm{op}} \to \mathsf{Set}$ should be thought of as a set of combinatorial gluing instructions for a topological space $|X|$. There is a \textit{\textbf{geometric realization functor}} given by $X \mapsto |X|$; see ~\cite{CGP1}, ~\cite{K}, or ~\cite{kannan2020topology} for a description of this functor.
    
    The symmetric $\Delta$-complex description of $\Delta_{g, w}$ is as follows: for each $p \geq 0$, we let \[\Delta_{g, w}([p]) = \{(\G, \tau) \in \pi_0(\Gamma_{g, w}^{\mathrm{EL}}) \mid |E(\G)| = p + 1 \}, \]
    where $\pi_0$ denotes the set of isomorphism classes. We put $[\G, \tau]$ for the equivalence class of a $\Gamma_{g, w}^{\mathrm{EL}}$-object $(\G, \tau)$, and will hereafter shorten $\Delta_{g, w}([p])$ to $\Delta_{g,w}[p]$. Given an injection $\iota: [p] \to [q]$, we define $\iota^* = \Delta_{g, w}(\iota): \Delta_{g, w}[q] \to \Delta_{g, w}[p]$ as follows: if $[\G, \tau] \in \Delta_{g, w}[q]$, then $\iota^*[\G, \tau]$ is the edge-labelled graph obtained by contracting all edges in $\G$ which are not labelled by the image of $\iota$, and then taking the induced labelling of the remaining edges which preserves their $\tau$-ordering. 
    
    \subsection{Automorphisms of $\Delta_{g, w}$ and the filtration by number of vertices}
    An automorphism of $\Delta_{g, w}$ is a natural isomorphism $\Delta_{g, w} \to \Delta_{g, w}$. To unpack this, we will identify a generating set for the morphisms in the category $\mathrm{I}$. For $p \geq 0$, put
    \[ \SS_{p + 1} := \mathrm{Hom}_{\mathrm{I}}([p], [p]),  \]
    so $\SS_{p + 1}$ is the group of permutations of the set $\{0, \ldots, p\}$. Given $\alpha \in \SS_{p + 1}$, we write $\alpha^* = \Delta_{g, w}(\alpha)$. Next, for each $i \in [p + 1]$, we put $\delta^{i}: [p] \to [p + 1]$ for the unique order-preserving injection whose image does not contain the element $i$. We put $d_i : = \Delta_{g, w}(\delta^i)$. It is apparent that any morphism $\iota: [p] \to [q]$ in the category $\mathrm{I}$ can be factored as a sequence of maps of the form $\delta^i$, followed by some element of $\SS_{q + 1}$. 
    
    An automorphism of $\Delta_{g, w}$ can therefore be understood as the data of bijections \[\Phi = \{\Phi_p: \Delta_{g, w}[p] \to \Delta_{g, w}[p]\}_{p \geq 0}, \] such that the diagrams
    \begin{equation}\label{Symmetric}
    \begin{tikzcd}
    &\Delta_{g,w}[p] \arrow[r, "\Phi_{p}"] \arrow[d, "\alpha^*"] &\Delta_{g, w}[p] \arrow[d, "\alpha^*"]\\
    &\Delta_{g, w}[p] \arrow[r, "\Phi_p"]  &\Delta_{g, w}[p]
    \end{tikzcd} 
    \end{equation}
    and
    \begin{equation}\label{Simplicial}
    \begin{tikzcd}
    &\Delta_{g,w}[p + 1] \arrow[r, "\Phi_{p + 1}"] \arrow[d, "d_i"] &\Delta_{g, w}[p + 1] \arrow[d, "d_i"]\\
    &\Delta_{g, w}[p] \arrow[r, "\Phi_p"]  &\Delta_{g, w}[p]
    \end{tikzcd} 
    \end{equation}
    commute for all $\alpha \in \SS_{p + 1}$ and $i \in [p + 1]$. We shall suppress the subscript and write $\Phi[\G, \tau]$ for $\Phi_p[\G, \tau]$.
    
    \begin{notn}
    Suppose $(\G, \tau) \in \mathrm{Ob}(\Gamma_{g, w}^{\mathrm{EL}})$, and that we have
    \[\Phi[\G, \tau] = [\G', \tau']. \]
    Then, for any $\alpha \in \SS_{p + 1}$, we must have 
    \[\Phi[\G, \tau \circ \alpha ] = [\G', \tau' \circ \alpha] \]
    by the commutativity of (\ref{Symmetric}). So, the action of $\Phi$ on one edge-labelling of $\G$ determines the action on all edge-labellings. We use the notation $(\Phi\G, \Phi\tau) := (\G', \tau')$; the graph $\Phi \G$ is determined up to isomorphism in $\Gamma_{g, w}$.
    \end{notn}
    
    \begin{rem}
    The group $S_n$ acts on $\Gamma_{g, n}$: if we are given $\G = (G, h, m) \in \mathrm{Ob}(\Gamma_{g, n})$, we put $\sigma \cdot \G = (G, h, m \circ \sigma^{-1})$; in this way the edges and vertices of $\G$ and $\sigma \cdot \G$ are identified, so that whenever the marking $i \in \{1, \ldots, n\}$ is supported on vertex $v$ in $\G$, the marking $\sigma(i)$ is supported on $v$ in $\sigma \cdot \G$. A given permutation $\sigma$ preserves the subcategory $\Gamma_{g, w}$ if and only if $\sigma \in \Aut(K_w)$. This gives the action of $\Aut(K_w)$ on $\Delta_{g, w}$ by automorphisms: $\sigma \cdot [\G, \tau] = [\sigma \cdot \G, \tau]$.
    \end{rem}
    
    \begin{exmp}\label{ExceptionalCase}
    When $w = (w_1, w_2)$ with $w_1 + w_2 \leq 1$, there are only two stable graphs in $\Gamma_{1, w}$ with a positive number of edges: a single loop, where the single vertex supports both markings, and a pair of parallel edges, where each vertex supports one marking. Both of these graphs have a unique edge-labelling up to their automorphism groups, so we have \[|\Delta_{1, w}[0]| = |\Delta_{1, w}[1]| = 1 \]
    while $\Delta_{1, w}[p] = \varnothing$ for $p>1$, so the automorphism group is trivial in this case. The geometric realization is given by the quotient of a $1$-simplex by its automorphism group $S_2$. 
    \end{exmp}
    
    Following ~\cite{K}, we analyze the action of $\Aut(\Delta_{g, w})$ by showing that it preserves the subspace $\V^{i}_{g, w}$ of $\Delta_{g, w}$ parameterizing tropical curves with at most $i$ vertices. Each $\V^{i}_{g, w}$ is a subcomplex of $\Delta_{g, w}$, and we have
    \[ \V^{1}_{g, w} \subseteq \V^{2}_{g, w} \subseteq \cdots \subseteq \V^{2g - 2 + n}_{g, w} = \Delta_{g, w}.  \]
    The proof that $\Aut(\Delta_{g, w})$ preserves this filtration is very similar to the proof of ~\cite[Proposition 3.4]{K}, with some minor differences. Due to this similarity, we record the result here and relegate its proof to Appendix \ref{VertexFiltration}.
    \begin{thm}\label{VFiltrationPreserved}
    Let $\Phi \in \Aut(\Delta_{g, w})$. Then $\Phi$ preserves the subcomplexes $\V^{i}_{g, w}$ for all $i \geq 1$. 
    \end{thm}
    
    The next theorem follows from Theorem \ref{Reconstruction}. We omit the proof, as it is exactly the same as in the special case of $w = (1^{(n)})$, which is ~\cite[Theorem 1.5]{K}.
    
    \begin{thm}\label{RestrictionInjects}
    Fix $g \geq 0$ and a weight vector $w \in \Q^n \cap (0, 1]^{n}$, where $2g - 2 + \sum w_i > 0$. Then the restriction map
    \[\Aut(\Delta_{g, w}) \to \Aut(\V^{2}_{g, w})  \]
    is an injection.
    \end{thm}
    
    
    \section{Calculation of $\Aut(\Delta_{g, w})$ for $g \geq 1$}\label{ProofofAuts}
    Suppose $g,n \geq 1$ and $w \in \Q^n \cap (0, 1]^n$, where $2g - 2 + n \geq 3$. If $n = 1$, then $\Delta_{g,w} \cong \Delta_{g, 1}$, so Theorem \ref{Auts} specializes to the main result of ~\cite{K}. Thus we hereafter assume $n \geq 2$ (and when $g = 1$, $n\geq 3$). To prove Theorem \ref{Auts}, we first show that any $\Phi \in \Aut(\Delta_{g, w})$ preserves the $S_n$-orbit of a given simplex in $\V^{2}_{g, w}$.
    \subsection{$\Aut(\Delta_{g, w})$ preserves $S_n$-orbits in $\V^2_{g,w}$} We want to show that for any $[\G, \tau]$ in $\V^{2}_{g, w}$, the action of $\Phi \in \Aut(\Delta_{g, n})$ preserves the $S_n$-orbit of $[\G, \tau]$. It suffices to show this in the case where $[\G, \tau]$ is a facet of $\V^{2}_{g, w}$, i.e. $[\G, \tau] \in \V^{2}_{g, w}[g]$. The first step is to show that $\Aut(\Delta_{g, w})$ preserves the isomorphism class of the edge-labelled graph underlying such a facet, if we forget the marking function. This motivates the following definition.
    
    \begin{defn}\label{WeakIsomorphism}
	Given two objects $(\G, \tau), (\G', \tau')$ of $\Gamma_{g, w}^{\mathrm{EL}}$ with $\G = (G, h, m)$ and $\G' = (G', h', m')$, we say that $(\G, \tau)$ and $(\G', \tau')$ are \textit{\textbf{weakly isomorphic}} if there exists an isomorphism of weighted graphs
	\[
	\varphi: (G, h) \to (G', h')
	\]
	making the diagram
	\[\begin{tikzcd}
	&E(G') \arrow[dr, "\tau'"] \arrow[rr, "\varphi^*"] & &E(G) \arrow[dl, "\tau"']\\
	& &\lbrack p \rbrack &
	\end{tikzcd} \]
	commute. Such a map $\varphi$ is called a \textit{\textbf{weak isomorphism of pairs}}, and is denoted with a dashed arrow \[\varphi: (\G, \tau) \dashrightarrow (\G', \tau').\]
    \end{defn}
    
    The proof of the following lemma is the same as that of ~\cite[Proposition 5.3]{K}, and is thus omitted.
    
    \begin{lem}\label{WeakPreserved}
	Suppose $\Phi \in \Aut(\Delta_{g,w})$, and let $[\G, \tau] \in \V^{2}_{g, w}[g]$. Then, for any representatives $(\G, \tau), (\Phi\G, \Phi\tau)$, there exists a weak isomorphism
	\[\varphi: (\G, \tau) \dashrightarrow (\Phi\G, \Phi\tau). \]
    \end{lem}
    
    Now we work towards the proof that $S_n$-orbits of simplices in $\V^{2}_{g, w}[g]$ are preserved. For this we will need the following lemma. We adopt the convention that a $1$-cycle of a graph is a loop, and a $2$-cycle is a pair of parallel edges.
    
    \begin{lem}\label{cyclespreserved}
    Let $\Phi \in \Aut(\Delta_{g,w})$, and suppose $[\G, \tau] \in \Delta_{g, w}[p]$. Then:
    \begin{enumerate}[(a)]
        \item a subset $S \subseteq [p]$ indexes a $k$-cycle of $\G$ via $\tau$ if and only if it indexes a $k$-cycle of $\Phi\G$ via $\Phi\tau$;
        \item a subset $\{i,j\} \subseteq [p]$ indexes a pair of loops on the same vertex of $\G$ if and only if it indexes a pair of loops on the same vertex of $\Phi \G$.
    \end{enumerate}
    \end{lem}
    \begin{proof}
     We prove each part separately.
    \begin{enumerate}[(a)]
        \item When $k = 1$, the claim is true as an index $i$ labels a loop of $\G$ if and only if $\G/e_i$ has the same number of vertices as $\G$, and $\Phi d_i[\G, \tau] = d_i [\Phi \G, \Phi \tau]$. Since $\Phi$ preserves the number of vertices, it follows that $i$ must label a loop index of $\Phi \G$. Now $S$ labels a $k$-cycle of $\G$ if and only if, for all $i \in S$, the set $\delta_i(S)$ labels a $(k - 1)$-cycle of $d_i[\G, \tau]$ (here, $\delta_i: [p] \smallsetminus \{i\} \to [p - 1]$ is the unique map such that $\delta_i \circ \delta^i = \mathrm{id}$). Thus the claim follows by induction.
        \item For $i, j$ to label a pair of loops on the same vertex of $\G$, we must have $(i,j) \in \mathrm{Stab}_{\SS_{p + 1}}[\G, \tau]$. By the commutativity of (\ref{Symmetric}), we must also have $(i, j)\in \mathrm{Stab}_{\SS_{p + 1}}[\Phi\G, \Phi\tau]$, and by the first part of the lemma, $i$ and $j$ must both label loops of $\Phi \G$. So $\Phi \G$ must have an automorphism which exchanges the loops $i$ and $j$, but which fixes every other edge. If $|V(\Phi\G)| \geq 3$, this is only possible if $i$ and $j$ label loops on the same vertex of $\Phi \G$. When $|V(\Phi\G)| = 2$, the claim follows from Lemma \ref{WeakPreserved}, and the claim is clear when $|V(\Phi\G)| = 1$. \qedhere 
    \end{enumerate}
    \end{proof}
  
    Lemmas \ref{WeakPreserved} and \ref{cyclespreserved} give us a criterion for checking whether $\Aut(\Delta_{g,w})$ preserves the weak isomorphism class of a given simplex, and at the level of $\V^{2}_{g, w}$, Lemma \ref{WeakPreserved} means that $\Aut(\Delta_{g,w})$ acts on an edge-labelled stable graph by at most changing the marking function. We would like to show that this redistribution preserves the number of markings on each vertex. It is useful to introduce notation parameterizing the facets of $\V^{2}_{g, w}$.
     
    Fix a vertex set $\{v_1, v_2\}$. For two integers $k, \ell$ such that $k, \ell \geq 0$ and $k + \ell \leq g$, we fix $B^{k,  \ell}$ to be a graph with vertex set $\{v_1, v_2\}$, where $v_1$ and $v_2$ are connected by $g - (k +  \ell) + 1$ edges, so that $v_1$ supports $k$ loops while $v_2$ supports  $\ell$ loops. 
    
    By construction, $B^{k,  \ell}$ has genus $g$ and $g + 1$ edges, and we have graph isomorphisms $B^{k,  \ell} \cong B^{ \ell, k}$. Up to isomorphism, any facet of the subcomplex $\V^2_{g, w}$ is a marked, edge labeled version of $B^{k,  \ell}$ for some $k$,  $\ell$. Given a subset $A \subseteq \{1, \ldots, n\}$, we put $B^{k, \ell}_A$ for an $n$-marked version of $B^{k, \ell}$, where a vertex with $k$ loops supports the elements of $A$ and the other vertex supports the elements of $A^c$. We will use the boldface notation $\B^{k, \ell}_A$ when $B^{k, \ell}_A$ defines a $\Gamma_{g, w}$-object. Fixing choices of edge-labellings $\pi^{k, \ell}: E(B^{k, \ell}) \to [g]$, we put $[\B^{k, \ell}_A] = [\B^{k, \ell}_A, \pi^{k, \ell}]$ for the resulting simplex in $\V^{2}_{g, w}[g]$. Throughout the remainder of this section, we will tacitly change the choice of $\pi^{k, \ell}$ for a given $k, \ell$ if it is necessary. To prove that $\Aut(\Delta_{g, w}) \cong \Aut(K_w)$, it suffices to show that for any $\Phi \in \Aut(\Delta_{g, w})$, there exists a unique element $\sigma \in \Aut(K_w)$ such that
    \[\Phi[\B^{k, \ell}_A] = [\sigma \cdot \B^{k, \ell}_A] = [\B^{k, \ell}_{\sigma(A)}] \]
    for all $k, \ell, A$.
    
    Following ~\cite{K}, we now observe that for a fixed choice of $A$, there always exists some (not necessarily unique) permutation $\sigma_A \in S_n$ such that $\Phi[\B^{k, \ell}_A] = [\B^{k, \ell}_{\sigma_A(A)}]$:
    
    \begin{thm}\label{OrbitsPreserved}
    Suppose $g, n \geq 1$ with $2g - 2 + n \geq 3$ and $w \in \Q^n \cap (0, 1]^n$. Then for all $\Phi \in \Aut(\Delta_{g,w})$ and $\B_{A}^{k, \ell}$, there exists some $\Phi(A) \subseteq \{1, \ldots, n \}$ such that $|A| = |\Phi(A)|$ and $\Phi[\B^{k, \ell}_A] = [\B^{k, \ell}_{\Phi(A)}]$. Moreover, the choice of such $\Phi(A)$ is unique unless $(k, \ell) = (0, 0)$ and $|A| = n/2$.
    \end{thm}
    \begin{proof}
    The argument is exactly the same as ~\cite[Section 5.2]{K}, except in the case $(k, \ell) = (0,0)$. In this case, for a fixed $A \subseteq \{1, \ldots, n\}$, we let \[\mu(A) = |\{\G \in \pi_0(\Gamma_{g, w}) \mid |V(\G)| = 3,  \G \text{ has a }3\text{-cycle,}\leq \text{1 multiedge, contracts to }\B^{0,0}_A \}|. \]
    For any $A \subseteq \{1, \ldots, n \}$, it is straightforward to compute
    \[ \mu(A) = 2^{|A|} + 2^{n - |A|} + \text{Constant},  \]
    cf. ~\cite[Proposition 5.2]{K}. By Lemma \ref{WeakPreserved}, there exists some $C \subseteq \{1, \ldots, n \}$ such that 
    \[\Phi[\B_{A}^{0,0}] = [\B_C^{0,0}], \]
    Lemma \ref{cyclespreserved} implies that we must have $\mu(A) = \mu(C)$. In particular, we must have either $|A| = |C|$ or $|A| = n - |C|$. Then unless $|A| = n/2$, there is a unique choice between $\Phi(A) = C, C^c$ such that $\Phi[\B_{A}^{0,0}] = [\B_{\Phi(A)}^{0,0}]$ and $|A| = |\Phi(A)|$. If $|A| = n/2$, this choice is only determined up to complementation.
    \end{proof}

    Recalling that the $\B^{k, \ell}_A$ are precisely the facets of $\V^2_{g,w}$ we can see, using that $\Phi(d_i[\B^{k,\ell}_A]) = d_i(\Phi[\B^{k,\ell}_A]$), that Theorem \ref{OrbitsPreserved} extends to the preservation of $S_n$-orbits in $\V^2_{g, w}$ as a whole.
    
    \subsection{Finishing the proof of Theorem \ref{Auts}}
    First we deal with the case $n = 2$. In this case, we have $\Aut(K_w) \cong S_2$, and $g \geq 2$. Given $x \in \{1, \ldots, n \}$ we have that the graph
    \[\B^{g-1, 1}_x := \B^{g-1, 1}_{\{x\}} \]
    is always stable. Given $\Phi \in \Aut(\Delta_{g, w})$, there exists a unique element $\sigma \in S_2$ such that 
    \[\Phi [\B^{g-1, 1}_x] = [\B^{g-1, 1}_{\sigma(x)}] \]
    for $x = 1, 2$.
    \begin{prop}\label{length2}
    Suppose $n = 2$, $g \geq 2$, and $w \in \Q^2 \cap (0, 1]^2$. Then if $\Phi \in \Aut(\Delta_{g, w})$ fixes the simplices $[\B^{g-1, 1}_x]$ for $x = 1, 2$, then $\Phi|_{\V^{2}_{g, w}} = \mathrm{Id}|_{\V^2_{g, w}}$.
    \end{prop}
    \begin{proof}
    This proof follows from that of ~\cite[Proposition 5.13(b)]{K}: in that argument, no reference is made to $w$-unstable simplices in order to show that a given $w$-stable simplex is fixed.
    \end{proof}
    Given the result of Proposition \ref{length2}, we assume that $n \geq 3$ for the remainder of this section. Since $g \geq 1$, the graph \[\B^{0,0}_{x}:= \B^{0,0}_{\{x\}}\] is stable for each $x \in \{1, \ldots, n \}$. Theorem \ref{OrbitsPreserved} gives a unique permutation $\sigma \in S_n$ such that $\Phi[\B^{0,0}_{x}] = [\B^{0,0}_{\sigma(x)}]$ for all $x \in \{1, \ldots, n \}$. We will first show that $\sigma \in \Aut(K_w)$, and then that $\Phi$ acts as $\sigma$ on all elements of $\V^{2}_{g, w}[g]$. For the remainder of the paper, by an \textit{\textbf{expansion}} of a graph $\H \in \mathrm{Ob}(\Gamma_{g, w})$ we will mean a graph $\G$ with $|E(\G)| > |E(\H)|$ admitting a $\Gamma_{g, w}$-morphism to $\H$.
    \begin{lem}\label{PermutationisStable}
    Suppose $g \geq 1$, $n \geq 3$, and $w \in \Q^n \cap (0, 1]^n$, and suppose $\Phi \in \Aut(\Delta_{g, n})$. Let $\sigma \in S_n$ be the unique permutation such that
    \[\Phi[\B^{0,0}_x] = [\B^{0,0}_{\sigma(x)}] \]
    for all $x \in \{ 1, \ldots, n\}$. Then $\sigma \in \Aut(K_w)$.
    \end{lem}
    \begin{proof}
    We will show that for $S \subseteq \{1, \ldots, n \}$, we have that $w(S)> 1$ if and only if $w(\sigma(S)) > 1$. By switching $\sigma$ with $\sigma^{-1}$, it suffices to just prove one direction. We may suppose that some $S$ with $w(S) > 1$ exists, because otherwise $\Aut(K_w) \cong S_n$. Since $w(S) > 1$, the graph $\B^{g, 0}_{S^c}$ is stable. Moreover, the graph $\B^{g, 0}_{S^c}$ shares an expansion with $\B^{0,0}_{x}$ if and only if $x \in S^c$. Therefore, if we let $\Phi(S^c)$ be as determined by Theorem \ref{OrbitsPreserved}, we have $\Phi(S^c) = \sigma(S^c)$. This implies that the graph $\B^{g,0}_{\sigma(S^c)}$ is stable, i.e. that $w((\sigma(S^c))^c) = w(\sigma(S)) > 1$, as we wanted to show.
    \end{proof}
    
    Given the result of Lemma \ref{PermutationisStable}, Theorem \ref{Auts} is rendered equivalent to the following result.
    
    \begin{thm}\label{FinalReduction}
    Fix $g \geq 1, n \geq 3$, and say that $\Phi \in \Aut(\Delta_{g, w})$ fixes each simplex $[\B^{0, 0}_x]$. Then $\Phi$ fixes all of the simplices $[\B^{k, \ell}_A]$.
    \end{thm}
    
    Theorem \ref{FinalReduction} is further broken up into the three intermediate results Proposition \ref{00A}, Proposition \ref{k0A}, and Proposition \ref{klA}, whose statements and proofs make up the remainder of this section.
    
    \begin{prop}\label{00A}
    Fix $g \geq 1, n \geq 3$, and suppose $\Phi \in \Aut(\Delta_{g, w})$ fixes the simplices $[\B^{0, 0}_x]$ for all $x \in \{1, \ldots, n\}$. Then, $\Phi$ fixes the simplices $[\B^{0, 0}_A]$ for all $A \subseteq \{ 1, \ldots, n\}$.
    \end{prop}
   
    \begin{proof}
    If $|A|$ or $|A^c| = 0$, then $[\B^{0,0}_A]$ is fixed by by Theorem 3.6.  If $|A|$ or $|A^c| = 1$ then $\B^{0, 0}_A$ = $\B^{0, 0}_x$ for some $x \in \{ 1, \ldots, n\}$, so it's fixed by hypothesis. So, assume that $|A|$ and $|A^c| \geq 2$, and for now we also assume that $|A| \neq n/2$, so in particular $|A| \neq |A^c|$ and there is a unique $\Phi(A) \subseteq \{1, \ldots, n \}$ with $|\Phi(A)| = |A|$ and $\Phi[\B^{0, 0}_A] = [\B_{\Phi(A)}^{0,0}]$. For each $x \in A$, consider the graph $\T_1$ of Figure \ref{B00A-fig1}, with some edge-labelling $\tau$ such that $d_i[\T_1, \tau] = [\B^{0,0}_x]$ and $d_j[\T_1, \tau] = [\B_{A}^{0, 0}]$; we put $[\T_1]:= [\T_1, \tau]$. Then, by Lemma \ref{cyclespreserved}, we have that $\Phi[\T_1]$ is weakly isomorphic to $[\T_1]$. Let $B_1, B_2, B_3 \subseteq \{1, \ldots, n\}$ be the markings on the vertices of $\Phi[\T_1]$ as indicated by Figure \ref{B00A-fig1}.
    
    \begin{figure}[h]
        \centering
        \includegraphics[scale=1.2]{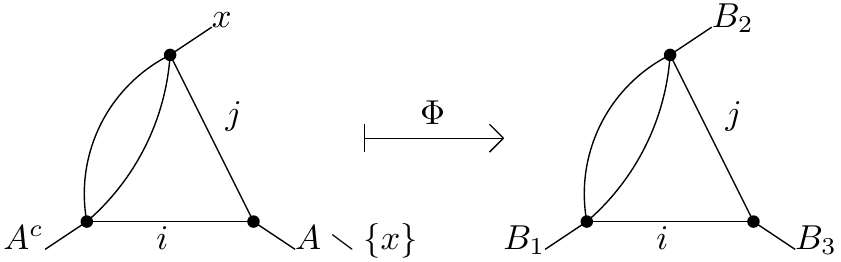}
        \caption{The simplex $[\T_1]$ and its image under $\Phi$}
        \label{B00A-fig1}
    \end{figure}
    
    Then, since $[\B_{x}^{0,0}]$ is fixed, we must have that either $B_1 \cup B_3 = \{x\}$ or $B_1 \cup B_3 = \{x\}^c$. But if $B_1 \cup B_3 = \{ x\}$, then by stability $B_3 = \{ x\}$ and $B_1 = \varnothing$. Upon contracting edge $j$ in both graphs, we would then have that the $S_n$-orbit of $[\B_{A}^{0,0}]$ is not preserved by $\Phi$, which is impossible. Therefore, we have $B_1 \cup B_3 = \{x\}^c$ and $B_2 = \{x\}$.  Moreover, since $S_n$-orbits of graphs with two vertices are preserved, we have $|A^c| + 1 = |B_1| + |B_2|$, so $|B_1| = |A^c|$, hence $|B_3| = |A| - 1$. It follows that $B_2 \cup B_3 = \Phi(A)$, hence $x \in \Phi(A)$. It follows that $A \subseteq \Phi(A)$, so in fact $A = \Phi(A)$ by Theorem \ref{OrbitsPreserved}.
    
    It is only left to show that $\Phi[\B^{0,0}_A] = [\B^{0,0}_A]$ when $|A| = n/2$. Since $n \geq 3$, this case only arises when $n \geq 4$ and $n$ is even. We treat the $n = 4$ case and the $n \geq 5$ case separately. If $n \geq 5$, then the set of graphs $\B_{S}^{0,0}$ such that $|S| = 2$ sharing an expansion with $\B_{A}^{0,0}$ are precisely those $S$ such that $S \subseteq A$ or $S \subseteq A^c$. Since $n \geq 5$, we know all of the $\B_{S}^{0,0}$ with $|S| = 2$ are fixed, hence for any choice of $\Phi(A)$ such that $\Phi[\B_{A}^{0,0}] = [\B^{0,0}_{\Phi(A)}]$, we must have that 
    \[ \binom{A}{2} \cup \binom{A^c}{2} =  \binom{\Phi(A)}{2} \cup \binom{\Phi(A)^c}{2}. \]
    This implies that $A = \Phi(A)$ or $A = \Phi(A)^c$, so in particular we have that $[\B_{A}^{0,0}]$ is fixed by $\Phi$.
    
    Finally consider the case when $n = 4$ and $|A| = 2$. Say $|A| = \{x, y\}$ and $A^c = \{u, v\}$, and suppose for sake of contradiction we have
    \[\Phi[\B_{\{x, y\}}^{0,0}] = [\B^{0,0}_{\{x, v\}}]. \]
    Then consider an expansion $\mathbf{T}_2$ of $\B_{\{x, y\}}^{0,0}$ as in Figure \ref{n=4case}, with an edge labelling $\tau$ so that $d_i[\mathbf{T}_2, \tau] = [\B^{0,0}_{\{x, y\}}]$ and $d_j[\mathbf{T}_2, \tau] = [\B^{0,0}_{x}]$; we put $[\mathbf{T}_2] = [\mathbf{T}_2, \tau]$. We let $B_1', B_2', B_3' \subseteq \{1, \ldots, n \}$ denote the marking sets on the vertices of $\Phi[\T_2]$ as in Figure \ref{n=4case}. 
    
    \begin{figure}[h]
        \centering
        \includegraphics[scale=1.2]{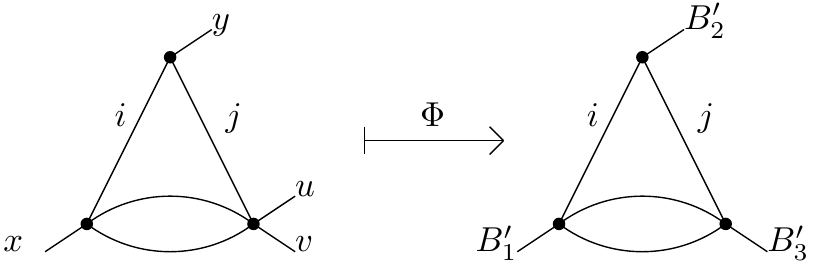}
        \caption{The simplex $[\mathbf{T}_2]$ and its image under $\Phi$}
        \label{n=4case}
    \end{figure}
    
    Then, since $\Phi[\B^{0,0}_{x}] = [\B^{0,0}_{x}]$, we must have $d_j\Phi[\mathbf{T}_2] = [\B^{0,0}_{x}]$, so either $B_2' \cup B_3' = \{x\}$ or $B_1' = \{x\}$. If $B_2' \cup B_3' = \{x\}$, then by stability we have $B_2' = \{x\}$ while $B_3' = \varnothing$, but then \[\Phi[\B_{\{x, y\}}^{0,0}] = d_i \Phi[\mathbf{T}_2] = [\B^{0,0}_{\varnothing}], \]
    which is a contradiction. Therefore $B_1' = \{x\}$ while $B_2' \cup B_3' = \{x\}^c$. Since $d_i \Phi[\mathbf{T}_2] = [\B_{\{x, v\}}^{0,0}]$ by hypothesis, it follows that $B_2' = \{v\}$ and $B_3' = \{u, y \}$. From this we may conclude that
    \[\Phi[\B^{0, g - 1}_{y}] = [\B^{0, g- 1}_{v}]. \]
    This clearly contradicts our hypothesis if $g = 1$, so we may now suppose that $g \geq 2$. 
    
    Finally, consider the simplex $[\mathbf{T}_3]$ and its image under $\Phi$ in Figure \ref{n=4casetwo}, where the edge-labelling is such that $d_i[\mathbf{T}_3] = [\B^{0,0}_y]$ and $d_j[\mathbf{T}_3] = [\B^{0, g-1}_y]$. Let $B_1'', B_2'', B_3'' \subseteq \{1, \ldots, n \}$ be the marking sets of $\Phi[\T_3]$ as in Figure \ref{n=4casetwo}.
    
    \begin{figure}[h]
        \centering
        \includegraphics[scale=1.2]{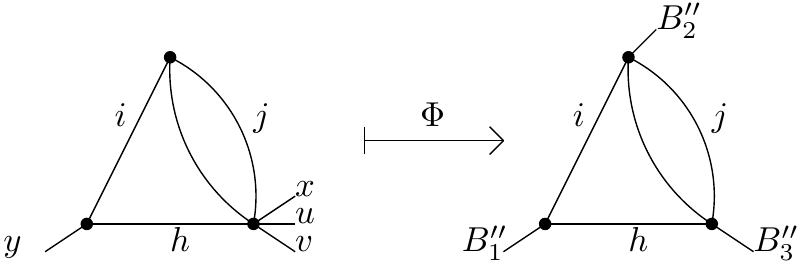}
        \caption{The simplex $[\mathbf{T}_3]$ and its image under $\Phi$}
        \label{n=4casetwo}
    \end{figure}
    
    Then, as we have $\Phi[\B^{0, g - 1}_{y}] = [\B^{0, g- 1}_{v}]$, it must be that $B_1''= \{v\}$ while $B_2'' \cup B_3 ''= \{v\}^c$. On the other hand, since $\Phi[\B^{0,0}_y] = [\B^{0,0}_y]$, we must have $d_i\Phi[\mathbf{T}_3] = [\B^{0,0}_y]$, hence $B_3'' = \{y\}$ and $B_2'' = \{x, v \}$. From this it follows that $\Phi d_h[\mathbf{T}_3] = \B^{0,0}_{\{x, v\}}$. This is a contradiction as $d_h[\mathbf{T}_3] = [\B^{0,0}_\varnothing]$ must be fixed by Theorem \ref{OrbitsPreserved}. This completes the proof. \qedhere
    

    \end{proof}
    
    The next step is to show that any automorphism which fixes the simplices $[\B^{0,0}_A]$ must also fix the simplices $[\B^{k, 0}_A]$.
    
    \begin{prop}\label{k0A}
    Fix $g \geq 1, n \geq 3$, and suppose $\Phi \in \Aut(\Delta_{g, w})$ fixes all simplices of the form $[\B^{0, 0}_A]$. Then, $\Phi$ fixes all simplices of the form $[\B^{k, 0}_A]$ with $k \leq g$.
    \end{prop}
    
    \begin{proof}
    If $k = 0$, then $[\B^{k,0}_A]$ is fixed by Proposition \ref{k0A}. When $k = g$, $\B^{g, 0}_A$ shares a common expansion with $\B_{x}^{0,0}$ if and only if $x \in A$, so $[\B^{g, 0}_A]$ is fixed because all of the simplices $[\B_{x}^{0,0}]$ are fixed. So suppose that $1 \leq k < g$ and that $A$ is nonempty.
    
    For a pair $[\B^{k,0}_A]$ and $[\B^{0,0}_x]$ with $x \in A$ consider the graph $\G_{k, A}$ of Figure \ref{k0A-fig}, where the $\{x, A - \{x\}\}$ multiedge has cardinality $k+1$ and contains an edge labelled $j$, and the $\{A^c, x\}$ multiedge has cardinality $g - k$ and contains an edge labelled $h$. Give such a graph an edge labeling $\tau$ such that \[d_i[\G_{k, A}] = [\B^{0,0}_x],\] \[d_j[\G_{k, A}] = [\B^{k,0}_A],\] and \[d_h[\G_{k, A}] = [\B^{g-k-1, 0}_{A^c \cup \{x\}}];\] we put $[\G_{k, A}] := [\G_{k, A}, \tau]$. Then, by Lemma \ref{cyclespreserved}, we  have that $\Phi[\G_{k, A}]$ is weakly isomorphic to $[\G_{k, A}]$. Let $C_1, C_2, C_3 \subseteq \{1, \ldots, n \}$ be the marking sets on the vertices of $\G_{k, A}$ as in Figure \ref{k0A-fig}.
    
    \begin{figure}[h]
    \centering
    \includegraphics[scale=1.2]{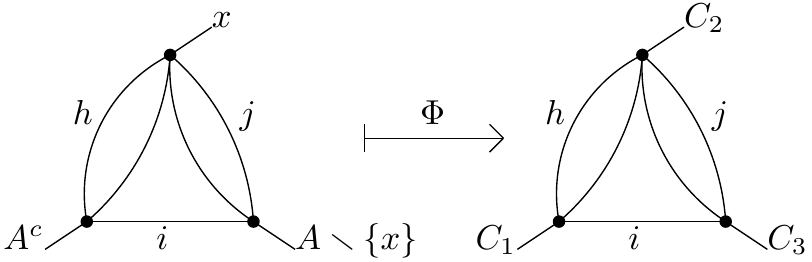}
    \caption{The simplex $[\G_{k, A}]$ and its image under $\Phi$}
    \label{k0A-fig}
    \end{figure}
    
     If we can show that $x \not\in C_1$, we're done: we must have $C_2 \cup C_3 = \Phi(A)$, where $\Phi(A) \subseteq \{1, \ldots, n \}$ is uniquely determined by Theorem \ref{OrbitsPreserved}. So if $x \in C_2 \cup C_3$ we have $x \in \Phi(A)$ which gives that $A \subseteq \Phi(A)$, and the result follows as $|A| = |\Phi(A)|$.
    
    As such, assume for contradiction that $x \in C_1$. Then $d_i[\G_{k, A}] = [\B^{0,0}_x]$, so $C_3 = \emptyset$, and $C_2 =\{x\}^c$. Since $n - 1 = |C_2 \cup C_3| = |\Phi(A)| = |A|$, this gives $|A^c| = |C_1| = 1$. Thus \[d_h\Phi[\G_{k, A}] = [\B^{g - k - 1, 0}_{\{1, \ldots, n\} }]. \] However, we have \[d_h[\G_{k, A}] = [\B^{g - k - 1, 0}_{A^c \cup \{x\}}],\] so we must have $n = |A^c| + 1 = 2$, a contradiction to our hypothesis that $n \geq 3$. \qedhere
    
    \end{proof}
    
    By symmetry, Proposition \ref{k0A} gives us that any $[\B^{0, \ell}_A]$ is also fixed.
    
    \begin{prop}\label{klA}
    Fix $g \geq 1, n \geq 3$, and suppose $\Phi \in \Aut(\Delta_{g, w})$ fixes all simplices of the form $[\B^{k, 0}_A]$ with $k < g$. Then, $\Phi$ fixes all simplices of the form $[\B^{k, \ell}_A]$.
    \end{prop}
     
    \begin{proof}
    If either of $k, \ell = 0$ then $[\B^{k,\ell}_A]$ is fixed by Proposition \ref{k0A}. Similarly, by Theorem \ref{OrbitsPreserved}, we may always assume $A, A^c$ nonempty. So, assume first that $k, \ell \geq 1$, $k + \ell < g$. For every pair $[\B^{k,\ell}_A]$ and $[\B^{0,\ell}_x]$ with $x \in A$ consider the graph $\G_{k,\ell,A}$ of Figure 6, where the $\{x, A - \{x\}\}$ multiedge has cardinality $k+1$, and the $\{A^c, x\}$ multiedge has cardinality $g - k - \ell$. Give such a graph an edge labeling $\tau$ such that $d_i[\G_{k,\ell,A}] = [\B^{0,\ell}_x]$, $d_j[\G_{k,\ell,A}] = [\B^{k,\ell}_A]$, and $d_h[\G_{k,\ell,A}]$ =  $[\B^{g-k-1,0}_{A^c\cup\{x\}}]$; we put $[\G_{k,\ell,A}] := [\G_{k,\ell,A}, \tau]$.
    \begin{figure}[h]
        \centering
        \includegraphics[scale=1.2]{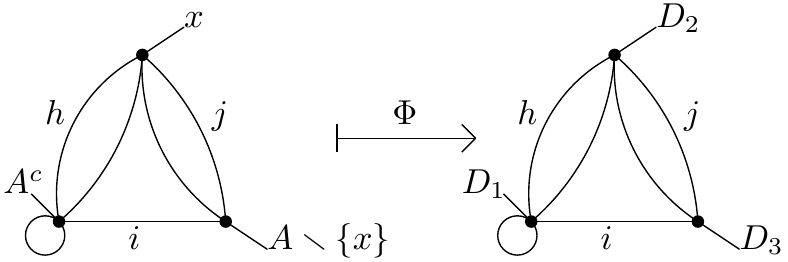}
        \caption{The simplex $[\G_{k, \ell, A}]$ and its image under $\Phi$}
        \label{klA-fig}
    \end{figure}
    By hypothesis, $\ell$ and $k$ are nonzero, so neither of the vertices adjacent to edge $j$ in $\Phi(\G_{k,\ell,A}]$ can support $\ell$ loops: if either did, then $d_j\Phi[\G_{k,\ell,A}] = [\B^{k+\ell,0}_{\Phi(A)}]$ for some choice of edge labeling, a contradiction to Lemma 3.2. This means that the $\ell$ loops in $\Phi[\G_{k,\ell,A}]$ are adjacent to edges $h$ and $i$. Thus, by Lemma 3.3 applied to all remaining edges of $\Phi[\G_{k,\ell,A}]$, we have that $\Phi[\G_{k,\ell,A}]$ is weakly isomorphic to $[\G_{k,\ell,A}]$, so the situation is as depicted in Figure \ref{klA-fig}; let $D_1, D_2, D_3 \subseteq \{1, \ldots, n \}$ be the marking sets on the vertices as in the figure.

    By Proposition 3.9, \[d_i\Phi[\G_{k,\ell,A}] = \Phi[\B^{0,\ell}_x] = [\B^{0,\ell}_x].\] Further, by the same, \[d_h\Phi[\G_{k,\ell,A}] = \Phi[\B^{g-k-1,0}_{A^c \cup \{x\}}] = [\B^{g-k-1,0}_{A^c \cup \{x\}}].\] Since $\ell \neq 0$, this implies that $D_2 = \{x\}$ and $D_1 \cup D_2 = A^c \cup \{x\}$, so $D_1 = A^c$, and thus $D_3 = A \smallsetminus \{x\}$. This implies that $\Phi[\G_{k,\ell,A}] = [\G_{k,\ell,A}]$, and so each contraction of $[\G_{k,\ell,A}])$ is fixed under $\Phi$ as well.
    
    Now say $k + \ell = g$, with $k, \ell > 0$. Recall that by Theorem 3.4, we may always assume that $|A|, |A^c| \geq 1$. Then, for every pair $[\B^{k,\ell}_A]$ and $[\B^{0,\ell}_x]$ with $x \in A$, consider the graph $\G'_{k,\ell,A}$ of Figure \ref{klA-fig2}, where the $\{x, A\smallsetminus\{x\}\}$ multiedge has cardinality $k+1$. 
    
    Give such a graph an edge-labelling $\tau$ such that $d_i([\G'_{k,\ell,A}]) = [\B^{0,\ell}_x]$ and $d_j([\G'_{k,\ell,A}]) = [\B^{k,\ell}_A]$; we put $[\G'_{k,\ell,A}] := [\G'_{k,\ell,A}, \tau]$. By Lemma \ref{cyclespreserved}, $\Phi([\G'_{k,\ell,A}])$ is weakly isomorphic to $[\G'_{k,\ell,A}]$: this is immediate ignoring the loops, while the loops cannot be adjacent to edge $j$ lest $d_j(\Phi([\G'_{k,\ell,A}])) = [\B^{g,0}_A]$ for some edge labeling of $\B^{g,0}_A$. Let $E_1, E_2, E_3 \subseteq \{1, \ldots, n\}$ be the markings on the vertices of $\Phi[\G_{k, \ell, A}']$ as in Figure \ref{klA-fig2}.
    
     \begin{figure}[h]
        \centering
        \includegraphics[scale=1.2]{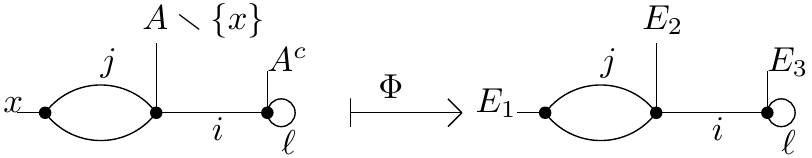}
        \caption{The simplex $[\G_{k, \ell, A}']$ and its image under $\Phi$}
        \label{klA-fig2}
    \end{figure}
    
    Then, one can see that $E_1 = \{x\}$ because 
    \[\Phi d_i[\G_{k, \ell, A}'] = \Phi [\B^{0, \ell}_x] = [\B^{0, \ell}_x] \]
    by Proposition \ref{k0A}. As such, since $E_1 \cup E_2 = \Phi(A)$, we have that $x \in A$ implies $x \in \Phi(A)$, and the result follows. \qedhere

    \end{proof}
    
    The preceding three propositions combine to prove Theorem \ref{FinalReduction}.
   \begin{proof}[Proof of Theorem \ref{FinalReduction}]
   Suppose given $\Phi \in \Aut(\Delta_{g, w})$ such that $\Phi$ fixes all the simplices $[\B_{x}^{0,0}]$. Proposition \ref{00A} shows then that $\Phi$ fixes all $[\B_{A}^{0,0}]$, then Proposition \ref{k0A} shows that $\Phi$ fixes the $[\B_{A}^{k,0}]$, and then Proposition \ref{klA} shows that $\Phi$ fixes all the $[\B_A^{k, \ell}]$.
   \end{proof}
   
   Finally, we conclude this section by indicating how Theorem \ref{Auts} follows from its results.
   
   \begin{proof}[Proof of Theorem \ref{Auts}]
   By Theorem \ref{RestrictionInjects}, it suffices to show that for any $\Phi \in \Aut(\Delta_{g, w})$, there exists a unique element $\sigma \in \Aut(K_w)$ such that $\Phi = \sigma$ when restricted to $\V^{2}_{g, w}$. Given such $\Phi$, there exists a unique permutation $\sigma \in S_n$ such that $\Phi = \sigma$ on the $n$ simplices $[\B_x^{0,0}]$, by Theorem \ref{OrbitsPreserved}. Lemma \ref{PermutationisStable} then implies that $\sigma \in \Aut(K_w)$. Then Theorem \ref{FinalReduction} gives that $\Phi \circ \sigma^{-1}$ acts as the identity on the facets of $\V^{2}_{g, w}$, from which it follows that $\Phi \circ \sigma^{-1}$ fixes the whole subcomplex $\V^{2}_{g, w}$. Thus $\Phi = \sigma$ on $\V^{2}_{g, w}$ and the proof is complete. 
   \end{proof}

\section{The genus 0 case}\label{genuszero}
When $g = 0$, Theorem \ref{Auts} fails for general $w$. We will first give some counterexamples, and then proceed to show that the theorem still holds when $g = 0$ and $w$ is assumed to be heavy/light.
\subsection{Counterexamples to Theorem \ref{Auts} when $g = 0$}
We now give an infinite family of counterexamples to Theorem \ref{Auts} in the case $g = 0$.
\begin{prop}\label{0dimcounterexamples}
For each integer $k \geq 1$, set $w^k = (1/k)^{(2k + 2)}$. Then \[\Aut(\Delta_{g, w^k}) \cong S_{N(k)},\] 
where \[ N(k) : = \frac{1}{2} \cdot \binom{2k + 2}{k + 1}. \]
Moreover, we have $N(k) > 2k + 2$ for $k \geq 2$ and $N(1) = 3 < 4$, so $\Aut(\Delta_{g, w^k}) \not \cong \Aut(K_{w^k})$ for all $k \geq 1$.
\end{prop}
\begin{proof}
 The space $\Delta_{0, w^k}$ consists of $N(k)$ disjoint vertices; this is because the only $w^k$-stable trees consist of only one edge, where each vertex supports $k + 1$ markings. This proves that $\Aut(\Delta_{0, w^k}) \cong S_{N(k)}$. Clearly $\Aut(K_{w^k}) \cong S_{2k + 2}$, so it only remains to prove that $N(k) > 2k + 2$ for all $k \geq 2$. We do this by induction. When $k = 2$, we have $N(2) = 10$ and $2k + 2 = 6$. Now suppose it is known that
 \[ \binom{2k}{k} > 4k. \]
 We then have
 \[ \binom{2k + 2}{k + 1} = \frac{(2k + 2)(2k + 1)}{(k+1)^2} \cdot \binom{2k}{k} > \frac{(2k + 2)(2k + 1)}{(k+1)^2} \cdot 4k. \]
 Observe that
 \[\frac{(2k + 2)(2k + 1)}{(k+1)^2} = \frac{4k^2 + 6k + 1}{k^2 +2k + 1} > 2 > 1 + \frac{1}{k}, \]
 for $k \geq 2$. Hence
 \begin{equation*}
 \binom{2k + 2}{k + 1} > \left(1 + \frac{1}{k} \right)4k = 4k + 4,
 \end{equation*}
 as desired.
 \end{proof}
 \begin{exmp}
 The family of examples provided by Proposition \ref{0dimcounterexamples} are $0$-dimensional. A $1$-dimensional example where $\Aut(\Delta_{0, w}) \not \cong \Aut(K_w)$ is given by $w = (1/3^{(3)}, 7/12^{(3)})$. In this case we have $\Aut(K_w) \cong S_3 \times S_3$, but $\Aut(\Delta_{0, w})$ is isomorphic to the wreath product $S_3 \wr S_3$. See Figures \ref{dim1counterexample} and \ref{weightcomplexS3S3}. This also gives an example where $\Aut(\Delta_{0, w})$ is not isomorphic to a direct product of symmetric groups, which cannot happen when $g \geq 1$, by Theorem \ref{Auts} and Theorem \ref{AllProductsArise}.
 \end{exmp}
 \begin{figure}[h]
     \centering
     \includegraphics[scale=1.2]{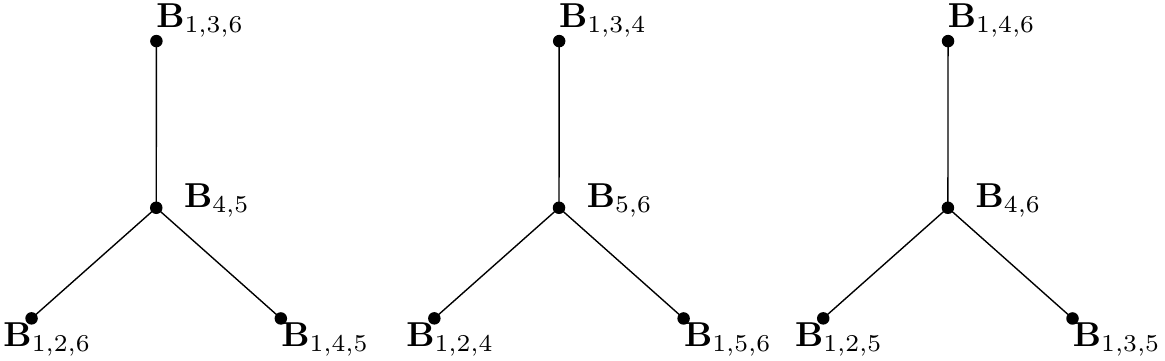}
     \caption{The tropical moduli space $\Delta_{0, w}$ for $w = (1/3^{(3)}, 7/12^{(3)})$.}
     \label{dim1counterexample}
 \end{figure}
 
  \begin{figure}[h]
     \centering
     \includegraphics[scale=1.2]{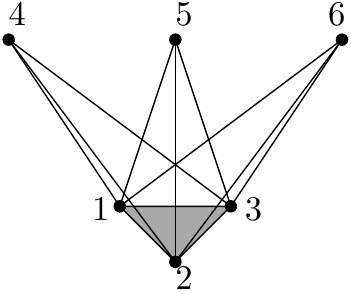}
     \caption{The simplicial complex $K_w$ for $w = (1/3^{(3)}, 7/12^{(3)})$.}
     \label{weightcomplexS3S3}
 \end{figure}
\subsection{Calculation of $\Aut(\Delta_{0, w})$ for heavy/light $w$}
In this section, we will remedy the genus $0$ failure of Theorem \ref{Auts} when $w$ is \textit{heavy/light}, i.e. when $w = (\varepsilon^{(m)}, 1^{(n)})$ for $\varepsilon \leq 1/m$.
\begin{thm}\label{AutsZero}
Suppose $m + n \geq 5$ and $n,m \geq 2$, and set $w = (\varepsilon^{(m)}, 1^{(n)})$ for any $\varepsilon \leq 1/m$. Then
\[\Aut(\Delta_{0, w}) \cong \Aut(K_w) \cong S_m \times S_n. \]
\end{thm}
To prove Theorem \ref{AutsZero}, we will describe $\Delta_{0, w}$ as a \textit{flag complex}, i.e. the maximal simplicial complex on its $1$-skeleton. This allows us to calculate $\Aut(\Delta_{0, w})$ by instead calculating the automorphism group of its $1$-skeleton.

\begin{rem}
In ~\cite{CHMR2014moduli}, the tropical moduli space $M_{0, w}^{\trop}$ is realized, for heavy/light $w = (\epsilon^{(m)}, 1^{(n)})$, as the Bergman fan $B(G_w)$ of the graphic matroid of the \textit{reduced weight graph} $G_w$ associated to $w$. The vertex set of $G_w$ is $\{1, 2, \ldots, m + n - 1\}$, and edge $(i, j)$ is included whenever $w_i + w_j > 1$. The space $\Delta_{0, w}$ can be constructed as the link of $M_{0, w}^\trop$ at its cone point, so in particular we have $\Aut(\Delta_{0, w}) \cong \Aut(B(G_w))$. The fan $B(G_w)$ carries actions of the groups $\Aut(G_w)$ and $\Aut(I(G_w))$, where $I(G_w)$ denotes the independence complex of the graph $G_w$. In general, we have $\Aut(G_w) \cong S_{m} \times S_{n - 1}$, while a general description of $\Aut(I(G_w))$ eludes the authors. In the case $n = 2$, the graph $G_w$ is a star with $m$ leaves, and $I(G_w)$ is a standard $(m-1)$-simplex, so we have that $\Aut(G_w) \cong \Aut(I(G_w)) \cong S_m$. By Theorem \ref{AutsZero}, in this case both groups are strictly smaller than the automorphism group $\Aut(B(G_w)) \cong S_m \times S_2$.
\end{rem}

\subsection{$\Delta_{0, w}$ as a flag complex}
When $g = 0$ and $w \in \Q^n \cap (0, 1]^n$ with $\sum w_i > 2$, the objects in $\Gamma_{0, w}$ are automorphism-free, and hence $\Delta_{0, w}$ may be realized as a simplicial complex. Given some $A \subseteq \{1,\ldots, n\}$ with $w(A), w(A^c) > 1$, we put $\mathbf{B}_A$ for a $\Gamma_{0, w}$-object with one edge, such that one vertex supports the elements of $A$, and the other supports the elements of $A^c$. A collection of vertices $\{\mathbf{B}_{A_1}, \ldots, \mathbf{B}_{A_k}\}$ spans a $(k - 1)$-simplex of $\Delta_{0, w}$ if and only if there exists a $\Gamma_{0, w}$-object $\G$ with precisely $k$ edges $e_1, \ldots, e_k$, such that
 \[ \G/\{e_i \}^c \cong \mathbf{B}_{A_i} \]
 for all $i$; here $\G/\{e\}^c$ indicates the graph obtained from $\G$ by contracting all edges except for $e$. 
 
 We claim that $\Delta_{0, w}$ is a flag complex. This claim when $w = (1^{(n)})$ is due to N. Giansiracusa \cite{Giansiracusa}, and its proof is based on the Buneman Splits-Equivalence Theorem ~\cite[Theorem 3.1.4]{Phylogenetics}, which we state here in a form compatible with our notation:
 \begin{thm}\label{Phylo}
    A collection $\{\mathbf{B}_{A_1}, \ldots, \mathbf{B}_{A_k}\}$ of vertices of $\Delta_{0, n}$ spans a simplex if and only if each pair $\{\mathbf{B}_{A_i}, \mathbf{B}_{A_j} \}$ forms a $1$-simplex of $\Delta_{0, n}$.
 \end{thm}
 The analogous theorem for $\Delta_{0, w}$ follows from the following observation: a graph $\G$ in $\Gamma_{0, n}$ lies in $\Gamma_{0, w}$ if and only if $\G /\{e\}^c$ lies in $\Gamma_{0, w}$ for all $e \in E(\G)$. Thus, if we are given a collection $\{\mathbf{B}_{A_1}, \ldots, \mathbf{B}_{A_k}\}$ of vertices of $\Delta_{0, w}$ such that each pair of them spans a $1$-simplex of $\Delta_{0, w}$, then we can use Theorem \ref{Phylo} to guarantee that there exists some graph $\G$ in $\Gamma_{0, n}$ such that $\G$ has precisely $k$ edges $e_1, \ldots, e_k$, and such that $\G / \{e_i \}^c \cong \mathbf{B}_{A_i}$ for all $i$. By our observation, we actually have $\G \in \mathrm{Ob}(\Gamma_{0, w})$, hence $\{\mathbf{B}_{A_1}, \ldots, \mathbf{B}_{A_k}\}$ spans a simplex of $\Delta_{0, w}$. As such, we have the following corollary of Theorem \ref{Phylo}.
 \begin{cor}
 The space $\Delta_{0, w}$ is a flag complex. In particular, we have
 \[\Aut(\Delta_{0, w}) \cong \Aut(\Delta_{0, w}^{(1)}), \]
 where $\Delta_{0, w}^{(1)}$ denotes the $1$-skeleton of $\Delta_{0, w}$.
 \end{cor}

\subsection{Calculation of $\Aut(\Delta_{0, w})$ for heavy/light $w$}
We now prove the following theorem.
\begin{thm}\label{MtimesN}
Let $m, n \geq 2$ such that $m + n \geq 5$. Then, if $w = (\varepsilon^{(m)}, 1^{(n)})$ for $\varepsilon \leq 1/m$, we have
\[\Aut(\Delta_{0, w}^{(1)}) \cong S_m \times S_n.  \]
\end{thm}

To prove Theorem \ref{MtimesN}, we will show that any automorphism of $\Delta_{0, w}^{(1)}$ can be completely described by its action on graphs of the form \[\B_{i, j} : = \B_{\{i, j\}}, \] where $i \in \{1, \dots m\}$ and $j \in \{m + 1, \dots, m + n\}$. Graphs of this form will be called \textit{\textbf{special}}. Special graphs have the maximal number of expansions among all graphs in $\Delta_{0, w}^{(1)}$:

\begin{lem}\label{maximize}
Consider the same hypotheses as Theorem \ref{MtimesN}. For a graph $\mathbf G$, let $\mathrm{exp}(\mathbf G)$ denote the number of isomorphism classes of expansions of $\mathbf G$ with precisely one more edge than $\G$. Then for all graphs $\mathbf B_{i, j}$ as above and for all vertices $\mathbf B_A \in \Delta^{(1)}_{0, w}$,
\[\mathrm{exp}(\mathbf B_{i,j}) \ge \mathrm{exp}(\mathbf B_A),\]
with equality if and only if $\mathbf B_A = \mathbf B_{i',j'}$ for possibly different indices $i' \in \{1, \dots m\}$ and $j' \in \{m + 1, \dots, m + n\}$.
\end{lem}

The proof of Lemma \ref{maximize} amounts to a somewhat tedious application of basic calculus, and can be found in Appendix \ref{Calculus}. Establishing an analogue of this lemma for arbitrary weight vectors seems to be the principal obstruction to determining the groups $\Aut(\Delta_{0, w})$ in general.

\begin{proof}[Proof of Theorem \ref{MtimesN}]
The desired isomorphism is given by the map
$$F: S_m \times S_n \to \Aut(\Delta_{0, w}^{(1)}),\quad (\sigma, \tau) \mapsto \Phi_{(\sigma, \tau)},$$ 
where $\Phi_{(\sigma, \tau)}$ is the automorphism of $\Delta_{0, w}^{(1)}$ that relabels light points using the permutation $\sigma \in S_m$ and the heavy points with the permutation $\tau \in S_n$. We must show that $F$ is both injective and surjective:

\subsubsection*{$F$ is injective} 
Supposing that $\Phi_{(\sigma, \tau)}$ acts as the identity on $\Delta_{0, w}^{(1)}$, we must show that $(\sigma, \tau)$ is the identity permutation. We use that $\Phi_{(\sigma, \tau)}$ in particular fixes each special graph $\mathbf B_{i,j}$. As we are assuming $m + n \ge 5$, the graph $\mathbf B_{i,j}$ has at least $3$ marked points on its other endpoint. It follows that $\{\sigma(i), \tau(j)\} = \{i, j\}$, or $\sigma(i) = i$ and $\tau(j) = j$. This demonstrates that $(\sigma, \tau)$ is the identity permutation.  

\subsubsection*{$F$ is surjective}
Fix an arbitrary automorphism $\Psi \in \Aut(\Delta_{0, w}^{(1)})$. Note that $\Psi$ preserves the number of expansions of a graph (i.e. the valence of a vertex in $\Delta_{0, w}^{(1)})$. This means that for all special graphs $\mathbf B_{i,j}$, we must have that $\Psi(\mathbf B_{i,j})$ is some special graph $\mathbf B_{i',j'}$ as well. Define the permutation $(\sigma, \tau) \in S_m \times S_n$ via $\sigma(i) := i'$ and $\tau(j) := j'$. We now claim that $\Phi_{(\sigma, \tau)} = \Psi$. Since $\Phi_{(\sigma, \tau)} \circ \Psi^{-1}$ fixes all special graphs by definition, it suffices to check that $(\Phi_{(\sigma, \tau)} \circ \Psi^{-1})(\mathbf B_A) = \mathbf B_A$ where $A$ is an arbitrary subset of left-hand weights. 

For any such graph $\mathbf B_A$, we can decompose $A$ into light and heavy weights as $A = A_L \sqcup A_H$, where $A_L \subset \{1, \dots, m\}$ and $A_H \subset \{m + 1, \dots, m + n\}$. Similarly we can decompose $A^C$ into disjoint sets $A^C_L$ and $A^C_H$. Note that the set of special graphs incident to $\mathbf B_A$ is incident is then precisely $\{\mathbf B_{i , j}\} \cup \{\mathbf B_{i',j'}\}$, where $(i, j) \in A_L \times A_H$ and $(i', j') \in A^C_L \times A^C_H$. 

In general if $\mathbf B_A$ is incident to special vertices $\{\mathbf B_{i,j}\}$, then $A$ can be recovered up to complement from the pairs $(i, j)$. Indeed, start with any such neighbor $\mathbf B_{i,j}$; without loss of generality, $i$ and $j$ are supported on the left-hand endpoint of $\mathbf B_{A}$. We can read off the rest of the markings on this vertex as follows. The left-hand light indices $i'$ are those for which $\mathbf B_{i',j}$ is incident to $\mathbf B_A$. Similarly, the left-hand heavy indices $j'$ are those for which $\mathbf B_{i',j'}$ is incident to $\mathbf B_A$ for all left-hand light indices $i'$. All of the weights $\{i'\} \cup \{j'\}$ either make up $A$ or $A^C$, so we conclude that $\mathbf B_A$ is uniquely determined by its special neighbors. 

In summary, we know that $\Phi_{(\sigma, \tau)} \circ \Psi^{-1}$ fixes all special neighbors of $\mathbf B_A$, and that $\mathbf B_A$ is the unique one-edge graph incident to all of these special neighbors. It follows that $\Phi_{(\sigma, \tau)} \circ \Psi^{-1}$ fixes $\mathbf B_A$ as well, so $\Phi_{(\sigma, \tau)} = \Psi$ and $F$ is surjective.
\end{proof}

\appendix
\section{Proof of Theorem \ref{AllProductsArise}}\label{ProofofAllProducts}
In this appendix we prove Theorem \ref{AllProductsArise}, restated here.
\begin{thm*}
    Let $G$ be a group. Then there exists $n \geq 1$ and $w \in \Q^n \cap (0, 1]^n$ such that
    \[\Aut(K_w) \cong G  \]
    if and only if $G$ is isomorphic to the direct product of finitely many symmetric groups.
\end{thm*}
    We will first prove that $\Aut(K_w)$ is always isomorphic to a product of symmetric groups, i.e. that it is generated by transpositions. We require a preliminary lemma.
    \begin{lem}
    Suppose $H$ is a subgroup of $S_n$, and that for all $\sigma \in H$ and $i \in \{1,\ldots, n\}$, we have $(i, \sigma(i)) \in H$. Then $H$ is generated by transpositions.
    \end{lem}
    \begin{proof}
    We want to show that given $\sigma \in H$, we can write $\sigma = \tau_1\cdots\tau_k$ where each $\tau_i \in H$ is a transposition. First consider the case where $\sigma = (i_1, \ldots, i_r)$ is a cycle. Then we have
    \begin{align*}
        \sigma &= (i_1, i_r)(i_1, i_{r-1}) \cdots (i_1, i_3)(i_1, i_2) \\&= (i_1, \sigma^{r-1}(i_1))(i_1, \sigma^{r-2}(i_1)) \cdots (i_1, \sigma^2(i_1))(i_1, \sigma(i_1)).
    \end{align*}
    Each transposition $(i_1, \sigma^{j}(i_1))$ lies in $H$, so the above gives a decomposition of the desired form for $\sigma$. To remove the assumption that $\sigma$ is a cycle, we decompose into disjoint cycles and run the same argument. 
    \end{proof}
   
    \begin{prop}\label{GeneratedByTranspositions}
    Let $w \in \Q^n \cap (0, 1]^n$. Then the subgroup $\Aut(K_w) \leq S_n$ is generated by transpositions. 
    \end{prop}
    \begin{proof}
    By the preceding lemma, it suffices to prove that if $\sigma \in \Aut(K_w)$ satisfies $\sigma(i) = j$, then $\tau = (i, j) \in \Aut(K_w)$. Indeed, suppose toward a contradiction that $\tau \notin \Aut(K_w)$. Then there exists $S \subseteq \{1, \ldots, n \}$ such that $S \in K_w$ but $\tau(S) \notin K_w$, i.e. $w(S) \leq 1$, but $w(\tau(S)) > 1$. If $i, j \in S$, or $i, j \in S^c$, then $w(S) = w(\tau(S))$ so it must be that exactly one of $i, j$ lies in $S$, suppose WLOG that $i \in S$ and $j \notin S$. Write \[S = \{\ell_1, \ldots, \ell_p, i\} = L \cup \{i\} \] where $L = \{\ell_1, \ldots, \ell_p \}$. For any natural number $k \geq 0$, we have $\sigma^k \in \Aut(K_w)$, so we must have \[w(\sigma^k(S)) = w(\sigma^k(L)) + w_{\sigma^k(i)} \leq 1,\] but using that $L = \tau(L)$ and $j = \sigma(i)$, we have \[w(\sigma^k(\tau(S))) = w(\sigma^k(\tau(L))) + w_{\sigma^k(j)} = w(\sigma^k(L)) + w_{\sigma^{k + 1}(i)} > 1, \]
    so in particular
    \[w_{\sigma^{k + 1}(i)} > w_{\sigma^k(i)}  \]
    for all $k \geq 0$. This is a contradiction as $\sigma$ has finite order. We conclude that $\tau \in \Aut(K_w)$, as we wanted to show.
    \end{proof}
    The following lemma allows us to symmetrize the weight data with respect to the action of $\Aut(K_w)$.
    
    \begin{lem}
    Suppose $n \geq 2$ and $w \in \Q^n \cap (0, 1]^n$. Then there exists some $\hat{w} \in \Q^n \cap (0, 1]^n$ such that 
    \begin{enumerate}[(i)]
        \item $K_{\hat{w}} = K_w$;
        \item if $\sigma \in \Aut(K_{\hat{w}})$ with $\sigma(i) = j$, then $\hat{w}_i = \hat{w}_j$.
    \end{enumerate}
    \end{lem}
    \begin{proof}
    Since $\Aut(K_w)$ is generated by transpositions, it suffices to show that if $\tau = (i, j) \in \Aut(K_w)$, then the weight vector $\hat{w}$ obtained from $w$ by changing both $w_i$ and $w_j$ to $(w_i + w_j)/2$ satisfies $K_{\hat{w}} = K_w$. Indeed, suppose $w(S) \leq 1$ to show that $\hat{w}(S) \leq 1$. If both $i, j$ are contained in $S$ or $S^c$, then $w(S) = \hat{w}(S)$, so it suffices to consider the case where $i \in S$ and $j \notin S$; write $S = L \cup \{i\}$ where $i,j \notin L$. Then
    \[\hat{w}(S) = w(L) + \frac{w_i + w_j}{2} \leq w(L) + \mathrm{max}(w_i, w_j) \leq 1,  \]
    since $\tau \in \Aut(K_w)$. This shows that any $S \in K_w$ satisfies $S \in K_{\hat{w}}$. Conversely suppose $\hat{w}(S) \leq 1$ to show that $w(S) \leq 1$. Again we may focus on the case where $i \in S$ but $j \notin S$; write $S = L \cup \{i\}$, where $i, j \notin L$. Suppose for contradiction that
    \[w(S) = \hat{w}(L) + w_i > 1. \]
    Then we must also have
    \[w(\tau(S)) = \hat{w}(L) + w_j > 1. \]
    It follows that
    \[\hat{w}(S) = \hat{w}(L) + \frac{w_i + w_j}{2} \geq \hat{w}(L) + \mathrm{min}(w_i, w_j) > 1,  \]
    which is a contradiction. Thus $S \in K_w$, and we are done.
    \end{proof}
    Proposition \ref{GeneratedByTranspositions} gives one direction of Theorem \ref{AllProductsArise}: since $\Aut(K_w)$ is generated by transpositions, it is always isomorphic to a direct product of symmetric groups, and this product has to be finite as $\Aut(K_w)$ is finite. We have left to show that an arbitrary finite direct product of symmetric groups can be realized in this way. 
\begin{proof}[Proof of Theorem \ref{AllProductsArise}]
    Suppose
    \[G \cong \prod_{i = 1}^{k} S_{n_i} \]
    for some integers $n_i \geq 1$. We prove that there exists $w$ such that $\Aut(K_w) \cong G$ by induction on $k$. When $k = 1$, we simply take $w$ to be an all $1$'s vector. For the inductive step, suppose we have some vector $\hat{w}$ such that $\Aut(K_{\hat{w}}) \cong \prod_{i = 1}^{k - 1} S_{n_i}$, in order to construct $w$ such that $\Aut(K_w) \cong G$. We may assume that $n_k > 1$.
    
    For an arbitrary vector $w$, we say an index $i \in \{1, \ldots, n \}$ is \textit{heavy in $w$} if $w_i + w_j > 1$ for all indices $j \neq i$. We say an index $i$ is \textit{light in $w$} if for all $S \subseteq \{1, \ldots, n \}$ with $w(S) < 1$, we have $w(S) + w_i \leq 1$. If $i$ is heavy, respectively light, in $w$, then we have that $(i, j) \in \Aut(K_w)$ if and only if $j$ is also heavy, respectively light, in $w$. Moreover, by the previous lemma, there exists some $\varepsilon > 0$ such that if $w'$ is obtained from $w$ by changing all heavy weights to $1$ and light weights to $\varepsilon$, then $K_{w'} = K_w$.
    
    If $\hat{w}$ does not contain any heavy, respectively light, weights, then we can construct $w$ by adding $n_k$ heavy, respectively light, weights to $\hat{w}$, in which case we have $\Aut(K_w) \cong G$. Otherwise, if $\hat{w}$ contains both heavy and light weights, we can assume all of the heavy weights are equal to $1$ and the light weights are equal to $\varepsilon$ for some $\varepsilon > 0$. Also by the previous lemma we may assume that whenever $\sigma \in \Aut(K_{\hat{w}})$ satisfies $\sigma(i) = j$, we have $\hat{w}_i = \hat{w}_j$. Then we set
    \[w = \left(\hat{w}_1, \ldots, \hat{w}_m, 1 - \frac{\varepsilon}{n_k}, \underbrace{\frac{\varepsilon}{n_k}, \ldots, \frac{\varepsilon}{n_k}}_{n_k} \right). \]
    We claim that $\Aut(K_w) \cong G$. Indeed, suppose that $\tau = (i, j) \in \Aut(K_{\hat{w}})$ and $S \subseteq \{1, \ldots, m + n_k + 1\}$ satisfies $w(S) \leq 1$. Then we claim that $w(\tau(S)) \leq 1$. Indeed, we have
    \begin{align*}
        w(\tau(S)) &= w(\tau(S \cap \{1, \ldots, m\})) + w(\tau(S \cap \{m +1, \ldots, m + n_k + 1 \})) \\& = \hat{w}(\tau(S \cap \{1, \ldots, m\})) + w(S \cap \{m +1, \ldots, m + n_k + 1 \}) \\&= \hat{w}(S \cap \{1, \ldots, m\}) + w(S \cap \{m +1, \ldots, m + n_k + 1 \}) \\&= w(S) \leq 1.
    \end{align*}
    Conversely, if $\tau = (i, j) \in \Aut(K_w)$ and $i, j \leq m$, then also $\tau \in \Aut(K_{\hat{w}})$. If $\tau = (i, j) \in \Aut(K_w)$ and $i > m + 1$, then we claim also $j > m + 1$. The weights $w_i$ for $i > m + 1$ are those which are equal to $\varepsilon/n_k$, and these are the unique light weights in $w$. Finally, we claim that there are no transpositions $\tau = (i, j) \in \Aut(K_w)$ where either $i$ or $j$ is equal to $m + 1$. This is because $m + 1$ is the unique vertex of $K_{w}$ which is connected by an edge to all of the light indices, but is not connected to any other indices: we have $\hat{w}_i \geq \varepsilon$ for all $i = 1, \ldots, m$.
    
    Altogether, this shows that
    \[\Aut(K_w) = \langle (i, j) \mid (i, j) \in \Aut(K_{\hat{w}}) \text{ or } w_i = w_j = \varepsilon/n_k  \rangle \cong \Aut(K_{\hat{w}}) \times S_{n_k} \cong G,  \]
    as desired.
    \end{proof}

\section{Proof of Theorem \ref{VFiltrationPreserved}}\label{VertexFiltration}

We now prove Theorem \ref{VFiltrationPreserved}, restated below. We reitirate that its proof is analogous to ~\cite[Proposition 3.4]{K}.
\begin{thm*}
    Let $\Phi \in \Aut(\Delta_{g, w})$. Then $\Phi$ preserves the subcomplexes $\V^{i}_{g, w}$ for all $i \geq 1$. 
\end{thm*}

Since $\Aut(\Delta_{g, w})$ preserves the number of edges of each edge-labelled graph, it suffices to show that $\Aut(\Delta_{g,w})$ preserves the first Betti number $b^1(\G)$ of the graph underlying a simplex $[\G, \tau]$. This is clear when $g = 0$, for in this case we have $b^1(\G) = 0$ for all $w$-stable graphs $\G$. Now fix $g \geq 1$, and $k$ such that $1 \leq k \leq g$. Put $\mathbf{R}_k$ for the unique (up to isomorphism) $\Gamma_{g,w}$-object with one vertex and $k$ loops. Each $\mathbf{R}_k$ has a unique edge-labelling up to the action of $\Aut_E(\mathbf{R}_k) \cong \SS_{k}$. We put $[\mathbf{R}_k] \in \Delta_{g, w}[k - 1]$ for the corresponding simplex of $\Delta_{g, w}$. Then given a simplex $[\G, \tau]$, we have that $b^1(\G) \geq k$ if and only if $[\G, \tau]$ has $[\R_k]$ as a face. Thus, to prove that $\Aut(\Delta_{g, w})$ fixes the first Betti number of each graph, it suffices to prove that it fixes each $[\R_k]$. Finally, each $[\R_k]$ is a face of $[\R_g]$, so it is enough just to show that $[\R_g]$ is fixed. We prove this in intermediate steps, the first being that the vertex $[\mathbf{R}_1] \in \Delta_{g, w}[0]$ is fixed.
    \begin{prop}\label{LoopPreserved}
    Suppose $g \geq 1$. Then for any $\Phi \in \Aut(\Delta_{g, w})$, we have $\Phi[\R_1] = [\R_1]$.
    \end{prop}
    \begin{proof}
    We say a graph $\G$ of $\Gamma_{g, w}$ is \textit{maximal} if the only graphs admitting morphisms to $\G$ are themselves isomorphic to $\G$. Edge-labellings of maximal graphs correspond to \textit{facets} of $\Delta_{g, w}$, where a \textit{facet} of a symmetric $\Delta$-complex is a simplex which is not a proper face of any other simplex. An automorphism of a symmetric $\Delta$-complex must permute the $d$-dimensional facets amongst themselves. We claim that $[\mathbf{R}_1] \in \Delta_{g, w}[0]$ is the unique vertex which is a face of all facets of $\Delta_{g, w}$. Graph-theoretically, this is equivalent to the statement that $\mathbf{R}^1$ is the unique graph in $\Gamma_{g, w}$ which has one edge and which admits a morphism from all maximal graphs. Indeed, any maximal graph $\G$ satisfies $b^1(\G) = g$, and so must have at least one cycle. We thus get a morphism $\G \to \mathbf{R}_1$ by contracting all edges except some fixed edge which is contained in a cycle of $\G$. To see that $\mathbf{R}_1$ is the unique graph with these properties, it suffices to exhibit a maximal graph $\G$ of $\Gamma_{g, w}$, such that if there exists a morphism $\G \to \H$ with $|E(\H)| = 1$, then $\H \cong \mathbf{R}_1$. When $g = 1$, such a graph $\G$ can be constructed by taking an $n$-cycle and putting one marking at each vertex. For each $g \geq 2$ there exists at least one graph $G$ such that:
    \begin{itemize}
        \item $G$ is trivalent;
        \item $b^1(\G) = g$;
        \item $G$ has no bridges
    \end{itemize}
    where a \textit{bridge} of a graph $G$ is a non-loop edge which is not contained in any cycles (see Figure 3 in ~\cite{K} for an example of such a graph in general). Then the necessary $\Gamma_{g, w}$ object $\G$ can be constructed by choosing $n$ points on the interiors of edges of $G$, and putting a vertex supporting a marking at each chosen point. The graph $\G$ cannot contract to a graph which has a bridge, so the only graph with one edge that it contracts to is $\R_1$. 
    \end{proof}
    
    To prove that the simplex $[\R_g]$ is preserved, we first preserve that bridge indices are preserved, as in the following lemma. 
    
    \begin{lem}\label{BridgesPreserved}
    Let $[\G, \tau] \in \Delta_{g, w}[p]$, and put $\mathcal{B}^{\G}_\tau \subseteq [p]$ for the indices of bridges of $\G$. Suppose that $\Phi \in \Aut(\Delta_{g, w})$ and $[\G', \tau'] = \Phi[\G, \tau]$. Then
    \[ \mathcal{B}^{\G}_\tau = \mathcal{B}^{\G'}_{\tau'}. \]
    \end{lem}
    \begin{proof}
    Given a simplex $[\G, \tau] \in \Delta_{g,w} [p]$ and a proper subset of indices $S \subset [p]$, we put $d_S[\G, \tau]$ for the face of $[\G, \tau]$ obtained by contracting all edges labelled by elements of $S$. From the commutativity of diagram \ref{Simplicial}, it can be shown that for any automorphism $\Phi$ of $\Delta_{g, n}$, we have $\Phi d_S[\G, \tau] = d_S \Phi[\G, \tau]$. With this notation in place, we can characterize $\mathcal{B}^{\G}_\tau$ as follows:
    \[\mathcal{B}^{\G}_\tau = \{ i \in [p] \mid d_{[p] \smallsetminus \{i\}} [\G, \tau] \neq [\R_1] \}. \] That is, an edge $e$ is a bridge of $\G$ if and only if upon contracting all edges in $\G$ besides $e$, we do not get a loop. The lemma now follows from this description of $\mathcal{B}^{\G}_\tau$ and Proposition \ref{LoopPreserved}.
    \end{proof}
    
    We can now prove that automorphisms preserve the simplex $[\R_g]$.
    \begin{prop}\label{RosePreserved}
    Let $g \geq 1$ and suppose $\Phi \in \Aut(\Delta_{g, w})$. Then $\Phi[\R_g] = [\R_g]$.
    \end{prop}
    \begin{proof}
    Suppose $\G$ is a maximal graph of $\Gamma_{g, w}$, with the property that every bridge of $\G$ is either a loop or a bridge (it is straightforward to construct examples of such $\G$ for all $g \geq 1$ and weight vectors $w$). Let $\tau: E(\G) \to [p]$ be any edge-labelling of $\G$, and put $[\G', \tau'] = \Phi[\G, \tau]$. Then we claim $\G'$ also has the property that all of its bridges are either loops or bridges. Indeed, $\G'$ must also be maximal, so $b^1(\G') = b^1(\G) = g$, and hence we have $|V(\G')| = |V(\G)|$. By Lemma \ref{BridgesPreserved}, $\G'$ has the same number of bridges as $\G$, and if we set $\mathcal{B} = \mathcal{B}^{\G}_\tau \subset [p]$, then $\mathcal{B}$ indexes the bridges in both $\G$ and $\G'$. Since bridges are contained in all spanning trees, the edges indexed by $\mathcal{B}$ in $\G'$ must be contained in some spanning tree of $\G'$. On the other hand, we know the edges indexed by $\mathcal{B}$ in $\G$ form a spanning tree of $\G$. Since $\G$ and $\G'$ have the same number of vertices, they have the same number of edges in a spanning tree. Therefore the edges indexed by $\mathcal{B}$ in $\G'$ form a spanning tree. Whenever we contract a spanning tree in a $\Gamma_{g, w}$-object of first Betti number $g$, the resulting graph is $\mathbf{R}_g$. In particular, we have
    \[ \Phi[\R_g] =  \Phi d_{\mathcal{B}} [\G, \tau] = d_{\mathcal{B}} [\G', \tau'] = [\R_g], \]
    and the proof is complete.
    \end{proof}
    As per the discussion preceding Proposition \ref{LoopPreserved}, Theorem \ref{VFiltrationPreserved} is a corollary of Proposition \ref{RosePreserved}.
\section{Proof of Lemma \ref{maximize}}\label{Calculus}   
We restate the lemma for convenience:
\begin{lemma}
Consider the same hypotheses as Theorem \ref{MtimesN}. For a graph $\mathbf G$, let $\mathrm{exp}(\mathbf G)$ denote the number of isomorphism classes of expansions of $\mathbf G$ with precisely one more edge than $\G$. Then for all graphs $\mathbf B_{i,j}$ as above and for all vertices $\mathbf B_A \in \Delta^{(1)}_{0, w}$,
\[\mathrm{exp}(\mathbf B_{i,j}) \ge \mathrm{exp}(\mathbf B_A),\]
with equality if and only if $\mathbf B_A = \mathbf B_{i',j'}$ for possibly different indices $i' \in \{1, \dots m\}$ and $j' \in \{m + 1, \dots, m + n\}$.
\end{lemma}

\begin{proof}
In what follows let $\mathbf B_A$ be a graph with one edge, where we think of the weights in $A$ as occupying the left-hand vertex. Set $x := |A|$, and suppose that there are $y$ light weights in $A$.

We are interested in maximizing the number of expansions of $\mathbf B_A$. Left expansions are bijective correspondence with subsets $S$ of $A$ such that $w(S) > 1$ and $|A \smallsetminus S| > 0$. There are $2^x - 2^y - (x - y) - 1$ such subsets $S$: $2^x$ total subsets of $A$, minus the $2^y$ subsets consisting solely of light weights (including the empty subset), minus the $x - y$ singleton subsets consisting solely of one heavy weight, minus the subset $A$ itself. Repeating the counting argument on the other side, we conclude that
\begin{align*}
\exp(\mathbf B_A) &= 2^x - 2^y - (x - y) - 1 + 2^{(m + n) - x} - 2^{m - y} - [(m + n - x) - (m - y)] - 1\\
&= 2^x - 2^y + 2^{m + n - x} - 2^{m - y} - 2 - n
\end{align*}
expansions total.

It therefore suffices to maximize
$$f(x, y) := 2^x - 2^y + 2^{(m + n) - x} - 2^{m - y}$$
over a domain that includes all permissible integer values of $(x, y)$. Such a domain is determined by the three inequalities $2 \le x \le (m + n) - 2$, $0 \le y \le m$, and $1 \le (x - y) \le n - 1$; see Figure \ref{fig:graph(6,2)}.

\begin{figure}[h]
    \centering
    \includegraphics{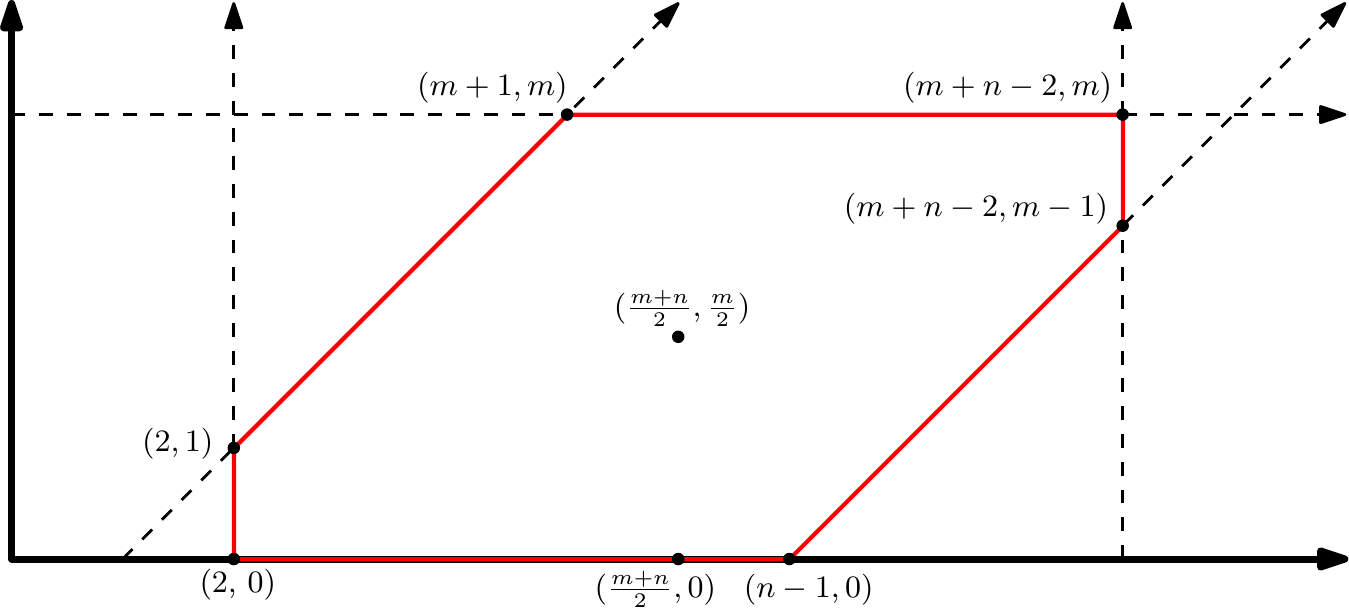}
    \caption{The domain of $f(x, y)$}  \label{fig:graph(6,2)}
\end{figure}

These inequalities arise as follows. First, the stability condition requires that both vertices of $\mathbf B_A$ have at least two weights on them, so $x = |A|$ is bounded by 2 and $m + n - 2$. Second, the number of light weights on either vertex cannot exceed $m$, the total number of light weights. Finally, there must be at least one heavy weight on either vertex, so that the number of left-hand heavy weights $x - y$ is at least 1 and at most $n - 1$. 

To prevent some of the above ranges from collapsing, it is convenient to address the case $n = 2$ separately from $n \ge 3$:
\subsubsection*{$\mathbf{n = 2}$.} As there are exactly two heavy weights, each vertex of $\mathbf B_A$ supports one of them. In other words, we have $y = x - 1$. We are now maximizing the function
$$f(x, x - 1) = 2^x - 2^{x - 1} + 2^{(m + 2) - x} - 2^{m - (x - 1)} = 2^{x - 1} + 2^{m - x + 1}.$$
over the interval $2 \le x \le m$. As the second derivative
$$\frac{d^2}{dx^2}(2^{x - 1} + 2^{m - x + 1}) = 2^{-x - 1}\log^2(2)(2^{m + 2} + 4^x)$$
is non-negative on $[2, m]$, $f(x, x- 1)$ is convex and thus achieves its maximum value on its endpoint $x = 2$ as desired. (The other endpoint $x = m$ corresponds to the complement $\mathbf B_{A^C} = \mathbf B_A$.)

\subsubsection*{$\mathbf{n \ge 3}$.} First, we look for critical points on the interior of the region. We compute the gradient as $\nabla f = \langle \log(2) 2^x - 2^{m + n - x}, 2^{-y} \log(2) (2^m - 4^y) \rangle$. Setting the partial derivatives equal to 0, we find that there is one critical point located at $((m + n) / 2, m / 2)$.

We now optimize $f$ over the boundary. Note that there is a symmetry originating from exchanging the two vertices of $\mathbf B_A$; symbolically, this is the involution $(x, y) \mapsto (m + n - x, m - y)$. Therefore it suffices to optimize $f$ over only half of the boundary, i.e. only over the left-hand equalities.

Specifically, we consider restricting $f$ to the following the three boundary segments:
\begin{align*}
&f(2, y),\ \text{for}\ 0 \le y \le 1,\\
&f(x, 0),\ \text{for}\ 2 \le x \le n - 1\\
&f(1 + y , y),\ \text{for}\ 1 \le y \le m.
\end{align*}
We first look for critical points on the interiors of these segments, and second consider the values of $f$ at their endpoints.

\begin{itemize}
    \item First, note that $f(2, y) = 2^{m + n - 2} - 2^{m - y} - 2^y + 4$ has a critical point at $y = m /2$. As $m \ge 2$, this critical point is outside of the interior of the interval for $y$.
    \item Second, note that $f(2, y) = f(x, 0) = 2^{m + n - x} - 2^m + 2^x - 1 + 4$ has one critical point at $x = (m + n) / 2$. This point is interior when $2 < (m + n) / 2$ and $(m + n) / 2 < n - 1$. That is, there is a critical point interior to this edge whenever $n > m + 2$.
    \item Finally, note that $f(1 + y, y) = 2^{m + n - y - 1} - 2^{m - y} + 2^y$ has a critical point when $2^m + 2^{2y} = 2^{m + n - 1}$. Since $m \neq m + n - 1$ and $2y = m + n - 1$ immediately leads to a contradiction, when $2y \neq m$ all of the exponents in this equation are distinct. Thus there is no integer solution by the uniqueness of binary representations. In case $2y = m$ the equation reduces to $m + 1 = m - n - 1$, or $n = 2$. As we are assuming $n \ge 3$, this edge contains no critical points. 
\end{itemize}

We now consider the function values at the interior critical point, the point on the boundary edge interior that is sometimes a critical point, and three of the six vertices:
\begin{align*}
    &f((m + n) / 2, m / 2) = 2^{(m + n)/ 2 + 1} - 2^{m/2 + 1}\\
    &f((m + n) / 2, 0) = 2^{(m + n) / 2 + 1} - 1 - 2^m\\
    &f(2, 0) = 2^{m + n - 2} - 2^m + 3\\
    &f(2, 1) = 2^{m + n - 2} - 2^{m - 1} + 2\\
    &f(n - 1, 0) = 2^m + 2^{n - 1} - 1.
\end{align*}

We claim that $f(2, 1)$ is at least as big as all of these values. It suffices to check the following four inequalities:

\begin{itemize}
\item $f((m + n) / 2, m / 2) < f(2, 1)$: We want to show that

\[2^{(m + n)/2 + 1} - 2^{m/2 + 1} < 2^{m + n - 2} - 2^{m - 1} + 2, \]
or equivalently that
\[2^{m - 1} - 2^{m/2 + 1} < 2^{m + n - 2} - 2^{(m + n)/2 + 1} + 2. \]
It suffices to show that the function
\[g(x) = 2^{x} - 2^{(x + 3)/2}  \]
is non-decreasing in $x$, for $x \geq 2$. Indeed we have
\[g'(x) = \log 2 \cdot 2^{x} - \frac{\log 2}{2} \cdot 2^{(x + 3)/2} = \log 2 \cdot 2^{x} - \log 2 \cdot 2^{1/2} \cdot 2^{x/2} > 0  \]
for $x > 1$, proving the claim.

\item $f((m + n) / 2, 0) < f(2, 1)$:
This inequality is equivalent to 
$$2^{(m + n)/2} < 2^{m + n - 3} + 2^{m - 2} + 3/2.$$
For $m + n \ge 6$ we have 
$$2^{(m + n)/2} \le 2^{m + n - 3} < 2^{m + n - 3} + 3/2,$$
proving the claim in every case except when $m + n = 5$. In that case, the inequality becomes 
$$2^m > 16\sqrt{2} - 22 \approx 0.6237,$$
which holds since $m \ge 2$.

\item $f(2, 0) < f(2, 1)$: This is equivalent to $3 - 2^{m} < 2 - 2^{m - 1}$, which is true for $1 < 2^{m - 1}$, i.e. for all $m > 1$.

\item $f(n - 1, 0) < f(2, 1)$: This is equivalent to 
$$2^{m - 1} + 2^m + 2^{n - 1} < 2^{m + n - 2} + 3,$$
or $$(2^m - 2)(2^n - 6) > 0.$$
As we have $m > 1$, the inequality reduces to $2^n > 6$, or $n > 2.58496...$. But $n \ge 3$, so this inequality holds.\qedhere
\end{itemize}
\end{proof}

\bibliographystyle{alpha}
\bibliography{hassettbib}
\end{document}